\documentclass[12pt]{article}

\usepackage{geometry}
\usepackage{amsmath}
\usepackage{amssymb}
\usepackage{amsthm}
\usepackage[latin1]{inputenc}
\usepackage{eurosym}
\usepackage[dvips]{graphics}
\usepackage{graphicx}
\usepackage{epsfig}
\usepackage{hyperref}

\usepackage{ifthen}

\makeindex

\newcommand{\supp}{\operatorname{supp}}

\newcommand{\Rr}{{\mathbb{R}}}

\newcommand{\Nn}{{\mathbb{N}}}

\newcommand{\Tt}{{\mathbb{T}}}

\newcommand{\Hh}{{\overline{H}}}

\newcommand{\Gg}{{\mathcal{G}}}

\newcommand{\ep}{{\epsilon}}

\newcommand{\Rm}{{\noindent \textbf {Remark:} \ }}

\theoremstyle{definition}

\newtheorem{theorem}{Theorem} \newtheorem{corollary}{Corollary}
\newtheorem{lemma}{Lemma} \newtheorem{proposition}{Proposition}
\newtheorem{definition}{Definition}

\theoremstyle{definition}

\newcommand{\nat}{I\!\!N} 
\newcommand{\z}{\mathbb{Z}} \newcommand{\re}{\mathbb{R}}
\newcommand{\tn}{\mathbb{T}^N}
\newcommand{\rn}{\mathbb{R}^N}
\newcommand{\zn}{\mathbb{Z}^N}

\begin{document}


\title{The Mather measure and a Large Deviation Principle for
the Entropy Penalized Method }
\author{D. A.  Gomes, A. O. Lopes and J. Mohr}
\date{\today}
\maketitle

\thanks{A. O. Lopes was partially supported by CNPq, PRONEX -- Sistemas
Din\^amicos, Instituto do Mil\^enio, and is beneficiary of
CAPES financial support,
J. Mohr was partially supported by CNPq PhD scholarship,
D. Gomes was
partially supported by the Center for Mathematical
Analysis, Geometry and Dynamical Systems through FCT Program POCTI/FEDER
and by grant POCI/FEDER/MAT/55745/2004.
}

\begin{abstract}

We present the rate function and a large deviation principle for the
entropy penalized Mather problem when the Lagrangian is generic (it is known that in this case the
Mather measure $\mu$ is unique and the support of $\mu$ is the Aubry set).
We assume the Lagrangian
$L(x,v)$, with $x$ in the torus $\mathbb{T}^N$ and $v\in \Rr^n$, satisfies certain natural hypothesis, such as superlinearity and convexity in $v$, as well
as some technical estimates.
Consider,
for each value of $\epsilon $ and $h$, the
%
%
entropy penalized Mather problem
\[\min \{ \int_{\tn\times\rn}
   L(x,v)d\mu(x,v)+\epsilon S[\mu]\},\]
 where the entropy $S$ is given by
$S[\mu]=\int_{\tn\times\rn}
\mu(x,v)\ln\frac{\mu(x,v)}{\int_{\rn}\mu(x,w)dw}dxdv,
$
and the minimization is performed over the space of probability
densities $\mu(x,v)$ on $\mathbb{T}^N\times\mathbb{R}^N$ that satisfy the
discrete holonomy constraint
$
\int_{\Tt^n\times \Rr^n}
\varphi(x+hv) -\varphi(x) d\mu=0
$.
It is known \cite{GV} that there exists a unique
minimizing measure
$\mu_{\epsilon, h}$ which converges to a Mather measure $\mu$, as $\epsilon, h\to 0$. In the
case in which the Mather measure $\mu$ is unique we prove a Large Deviation
Principle for the limit $\lim_{\epsilon,h\rightarrow0}\,\, \epsilon\,
\ln
    \mu_{\epsilon,h}(A),$ where $A\subset \mathbb{T}^N\times\mathbb{R}^N$.
In particular, we prove that
the deviation function $I$ can be written as  $I(x,v)=
L(x,v)+\nabla\phi_0(x)(v)-\overline{H}_{0},$ where $\phi_0$ is the
unique viscosity solution of the Hamilton-Jacobi equation, $H(\nabla\phi(x),x)=-\Hh_0 $.
We also prove a large deviation principle for the limit $\epsilon\to 0$ with fixed $h$.

Finally,  in the last section, we study some dynamical
properties of the discrete time Aubry-Mather problem, and
present a proof of the existence of a separating subaction.
\end{abstract}

\section{Introduction}

Recently, several results concerning large deviations as well as asymptotic limits
for Mather measures have appeared in the literature (see, for instance
\cite{A1},
\cite{A2}, \cite{AIPS}, \cite{BLT}). In this paper
we will consider a related
setting: the entropy penalized method introduced in
\cite{GV}. We study the rate of convergence of the entropy penalized Mather measures by establishing several large deviations results.

Let $\mathcal M$ denote the set of probability measures on $\tn\times\rn$.

The Mather problem  (see \cite{Mat},  \cite{Man}, \cite{CI} and    \cite{Fa}) consists in determining
probability measures $\mu\in \mathcal M$, called Mather measures,
which minimize the action
 \begin{equation} \label{Ma} \int_{\tn\times\rn} L(x,v)\,d\mu(x,v),
 \end{equation}
 among the probabilities $\mu \in \mathcal M$ that are invariant by the Euler-Lagrange flow for $L$. The Mather  measures usually are not absolutely continuous with respect to Lebesgue measure and are supported in sets which are not attractors for the flow. In this way, given $L$, it is important to have computable methods that permit, in some way,  to show the approximate location of the support of these measures.

 For $h>0$ fixed,
in analogy with the continuous case,  define the set of discrete holonomic measures as
\begin{equation}\label{constraint}
\mathcal M_h:=\left\{\mu\in\mathcal M;   \int_{\Tt^N\times \Rr^N}
\varphi(x+hv) -\varphi(x) d\mu=0, \forall \varphi\in C(\tn) \right\}.
\end{equation}
 Any measure $\mu\in\mathcal M_h$ is called
a discrete holonomic  measure. We denote by $\mathcal M_h^{a.c.}$ the measures in $\mathcal M_h$ which admit a density.

The discrete time
Aubry-Mather problem, see \cite{Gom},  consists in determining
probability measures $\mu\in \mathcal M_h$
  that minimize the action
 \begin{equation} \label{action} \int_{\tn\times\rn} L(x,v)\,d\mu(x,v).
 \end{equation}
Motivated by the papers \cite{A1}, \cite{A2},
the entropy penalized method was introduced
in \cite{GV} in order to approximate Mather measures by smooth densities.
The entropy penalized Mather problem, for $\ep>0$ and  $h>0$ fixed, consists in
\[
\min_{\mathcal M_h^{a.c.}} \left\{ \int_{\tn\times\rn}
   L(x,v)\,d\mu(x,v)+\epsilon S[\mu]\right\},
\]
 where the entropy $S$ is given by
\[
S[\mu]=\int_{\tn\times\rn}
\mu(x,v)\ln\frac{\mu(x,v)}{\int_{\rn}\mu(x,w)dw}dxdv,
\]
The entropy penalized method can be seen as a procedure to
approximate Mather measures by absolutely continuous probability
measures. These measures can be
obtained as a fixed point of an operator ${\cal G}$,  to be described later, from a
discrete time process with small parameters $\epsilon, h$. Furthermore, this fixed point can be
obtained by means of iteration of the operator
${\cal G}$.
In \cite{GL1} it is shown that, for $\epsilon$ and $h$ fixed,  the
velocity of convergence to the fixed point is exponentially fast.

In this paper we assume that the Lagrangian $L$ is such that the Mather measure is unique.
Then, it follows from a result by D. Gomes and E. Valdinoci \cite{GV} that
$\mu_{\epsilon, h}$ (the solution of the entropy penalized problem)
converges to
a discrete  Mather measure $\mu_h$, i.e., a measure that minimize \eqref{action} over $\mathcal M_h$.
Furthermore, by
 a result of D. Gomes (see \cite{Gom} and \cite{CDG}), with the Lagrangian satisfying some hypothesis to be stated  in the next section,
 the sequence of measures $\mu_h$ converges, through a subsequence, to the Mather measure $\mu$. Hence $\mu_{\epsilon, h}$ converges, through a subsequence, to  $\mu$.

 We address here the question of estimating how good is this approximation. In this way, it is natural to consider a Large Deviation Principle (L.D.P. for short)  for such limit. We refer the reader  to \cite{DZ} for general properties of  large deviation theory.

We start in the next section
by describing briefly the entropy penalized
Mather measure problem, as well as stating some of the results, such as the
uniform
semiconcavity estimates, that we will need throughout the paper. We
refer the reader to \cite{CS} for general results concerning semiconcavity.
In this section we also
generalize a result by D. Gomes and E. Valdinoci which shows the existence, for each $\ep$ and $h$,  of
 a density of  probability $\mu_{\epsilon,h}$
on $\Tt^N\times \Rr^N$
which solves the entropy penalized Mather problem.
This generalization is essencial for the large deviation results later in the paper.

In the two next sections we consider Large Deviation Principles in the following three  forms:

Firstly, for $h$ fixed,  as $\mu_{\epsilon,h}\rightharpoonup \mu_{h}$, we show the existence of a rate function $I_h$  such that,

 (a) If $A\subset \mathbb{T}^N \times\rn$ is a closed (resp. open) and bounded set, then
  $$\label{dinf1}\lim_{\epsilon\rightarrow0}\epsilon\ln
    \mu_{\epsilon,h}(A)\leq
    -\inf_A I_h (x,v) \;\; (\mbox{ resp. } \geq)
  $$

In order to do prove this result, we also need to study
some dynamical properties
of the discrete time Aubry-Mather problem, namely, the uniqueness of the calibrated subaction
for the discrete time problem. Because of its independent interest, we present these results in a separate section in the end of the paper.

\smallskip

For our second large deviation result,
we assume   that the Mather measure is unique and  the support of this measure is the Aubry set, hence there exists only one viscosity solution, say $\phi_0$.
Then, as $\mu_{\epsilon,h}\rightharpoonup \mu$,

(b) If $A\subset \mathbb{T}^N \times\rn$ is a  closed (resp. open) and bounded set such that $\pi_1(A)\cap \mathcal A\neq \emptyset$, where  $\mathcal A$ is the projected Aubry set, then there exists a function
$I(x,v)$  such that
  $$\label{dinf1}\lim_{\epsilon,h\rightarrow0}\epsilon \,  \ln
    \mu_{\epsilon,h}(A)\leq
    -\inf_A I (x,v)\;\; (\mbox{ resp. } \geq)
  $$
In this  case  we show that the deviation function $I$ is given by
$$I(x,v)= L(x,v)+\nabla\phi_0(x)(v)-\overline{H}_{0},$$where $\Hh_0$ is the Ma\~n\'e's critical value.

\smallskip

We point out that we just consider $I(x,v)$ for the points $x$ where $\nabla\phi_0(x)$ is defined. For the others points $x$ we declare $I(x,v)=\infty$. We remark that  $\mu_{\epsilon,h}$ is absolutely continuous with respect to Lebesgue measure on the tangent bundle  $\tn\times\rn$, and, as $ \nabla\phi_0$ is Lipchitz, $ \nabla\phi_0$ is differentiable almost every where
in the compact manifold $\tn$ where lives the $x$ variable. In this way, all points $(x,v)$ we consider in the support of
$\mu_{\epsilon,h}$ are assumed to be such that $ \nabla\phi_0(x)$ is defined.
\smallskip

Finally, the last case is:

(c) If $A\subset \mathbb{T}^N \times\rn$ is a  closed (open) and bounded set such that $\pi_1(A)\cap \mathcal A= \emptyset$ we will show a l.d.p. which yields
an estimate for the convergence rate of
  $$\label{dinf1}\lim_{\epsilon,h\rightarrow0}\epsilon\,h  \,  \ln
    \mu_{\epsilon,h}(A).
$$

In the last section we study the discrete  Aubry-Mather problem under the point of view of subactions, i.e., continuous functions that satisfy \begin{equation} \label{subaction}u(x)-u(x+hv)\leq h(L(x,v) -\Hh_h) \;\;\;\;\forall\; (x,v)\in\tn\times\rn\end{equation} for each  $h>0$ fixed. Where  $\Hh_h$ is the analog of the Mañé's critical value, i.e.,
 $$\Hh_h=\min_{\mathcal M_h} \left\{ \int_{\tn\times\rn}
   L(x,v)d\mu(x,v)\right\}$$
  There exist two important classes of subactions in which we are interested. The first class is composed of the calibrated subactions, those such that
  $$ u(x)= \inf_{v \in \mathbb{R}^N } \, \{ u(x+h\, v) + h\, L(x,v) -  h\, \Hh_h \}.
$$
The second class of subactions consists in the separating subactions, that is,
those for which the equality in (\ref{subaction}) is  attainned for some $(x,v)$ if, and  only if   $x\in\Omega(L)$ (this set will be defined in the last section).

Under the hypothesis that the Lagrangian is generic, we will show that there exists only one calibrated subaction, which gives  the uniqueness of the deviation function $I_h$.
Furthermore, we will establish the existence of a separating subaction, which can be considered as discrete  analog of the main result of   \cite {FS}.

By the way, we point out that according to \cite{FS} we can add to the Lagrangian $L(x,v)$ a term $d \varphi$, where $\varphi$ is differentiable $C^2$, in such way that the Mather measures for  $\hat{L}= L +d\varphi$ are the same as for $L$, $\overline{H}_{0} $ is the same, etc..., and moreover
$$\hat{I}(x,v)= \hat{L}(x,v)+\nabla\phi_0(x)(v)-\overline{H}_{0}=0,$$
if an only if, $(x,v)$ is in the support of the Mather measure.

The last author would like to thanks Philippe Thieullen for interesting conversations on the subject of the paper.

\section{The entropy penalized Mather problem}
\label{epmp}

In \cite{GV}, the Lagrangian $L:\rn\times\rn\to \re$  has the form
$$L(x,v)=K(v)-U(x), \;\;\;\;\; \mbox{ for } v\in\rn, x\in\rn,$$
in which $K$ is strictly  convex in $v$ and superlinear at infinity, and
the potential energy $U$ is bounded, $\z^N$-periodic and semiconvex, that is, there exists
$C_U>0$ such that
 $$\inf_{x,y\in\rn,y\neq 0} \frac{U(x+y)+U(x-y)-2U(x)}{|y|^2}\geq -C_U.$$
 Furthermore, $K$ is semiconcave, i.e., that there exists
 $C_K$ such that
 $$\sup_{v,w\in\rn,w\neq 0} \frac{K(v+w)+K(v-w)-2K(v)}{|w|^2}\leq C_K.$$

In this work we will need to work in slightly generalized setting. The main reason is that even if the Lagrangian has the form $L(x,v)=K(v)-U(x)$,
the time-reversed Lagrangian $L(x+h v, -v)$ will not have this form in general. The time-reversed Lagrangian, however, arises naturally in our problems.
Therefore need to modify our hypothesis accordingly.

\medskip

We will assume in the whole paper  that the Lagrangian $L:\rn\times\rn\to\re$,  is $\z^N$-periodic (we can consider it as a function $L:\tn\times\rn\to\re$),
 and satisfies the following estimates:

 \begin{enumerate}
\item Uniform superlinearity:   $$\lim_{|v|\to \infty}\frac{L(x,v)}{|v|}=+\infty,\;\;\;\; \mbox{ uniformly on }x\in\tn.$$

\item Convexity in $v$: the Hessian matrix   $\frac{\partial^2L}{\partial v_i \partial
v_j}(x,v)$ is positive definite.

\item There exist uniform constants  $C,\Gamma>0$ such that
$$ L( x+ y, v - z) + L( x-  y, v + z) -
2\, L(x,v)
\leq C \,  |y|^2 \, +\,
\Gamma  |z|^2$$
\end{enumerate}

 We consider here, the optimal control setting, where
\[
H(p, x)=\sup_v ( -p\cdot v-L(x, v)).
\]

\bigskip

 \Rm In the Classical Mechanics setting,   we usually define the
Hamiltonian in a different way, that is
\[
H(p, x)=\sup_v (\, p\cdot v-L(x, v)).
\]
These two definitions differ  by the sign of  $p\cdot v$.  And they are related in the following way:
if, instead of $L(x,v)$, we begin with the symmetrical Lagrangian, i.e., $\check{L}(x,v)=L(x,-v)$ (see \cite{Fa} § 4.5), then
$$\check{H}(p,x)=\max_v \{p\cdot v - \check{L}(x,v)\}=\max_v \{-p\cdot v - L(x,v)\}$$
Therefore,
the results presented here also hold, of course, in the
Classical Mechanics setting of Aubry-Mather theory.

\bigskip

Consider,
for each value of $\epsilon $ and $h$, the following
operators acting on continuous functions $\phi:\tn\rightarrow\re$:
$$\Gg[\phi](x):= -\epsilon h \mbox{ ln}\left[
  \int_{\re^N} e ^{-\frac{hL(x,v)+\phi(x+hv)}{\epsilon h}}dv\right] ,$$
and
                     $$\bar \Gg[\phi](x):= -\epsilon h \mbox{ ln}\left[
  \int_{\re^N} e ^{-\frac{hL(x-hv,v)+\phi(x-hv)}{\epsilon
      h}}dv\right].$$

We point out that the $\epsilon$ in \cite{GV} correspond here to $\epsilon h.$

\Rm Let  $\bar L$ be the Lagrangian  given by $\bar L(x,v)=L(x+hv,-v)$, we have that $\bar \Gg $ is the operator $\Gg$ for the Lagrangian $\bar L$. Hence, it is enough to prove the properties we need  for $\Gg$.

\begin {theorem}\label{pfixo}
Suppose $L$ satisfies  assumptions (1) to (3) above.
Then
for $\epsilon$ and $h$ fixed there exist
  $\lambda_{\epsilon,h}\in\mathbb{R}$ and $\z^N$-periodic Lipschitz
  functions $\phi_{\epsilon,h}$, $\bar \phi_{\epsilon,h}$
  so that
\begin{equation}\label{G}\Gg[\phi_{\epsilon,h}]=\phi_{\epsilon,h}+\lambda_{\epsilon,h},
\end{equation}
and
\begin{equation}\label{barG}\bar \Gg[\bar \phi_{\epsilon,h}]=\bar \phi_{\epsilon,h}+\lambda_{\epsilon,h}.
\end{equation}
Also there exists a constant $\bar C$ such that  the semiconcavity modulus of
$\phi_{\epsilon,h}$ and $\bar \phi_{\epsilon,h}$ is bounded by $\bar C$ for all $\epsilon$ and all $h$ sufficiently small.
\end{theorem}
\begin{proof}
We need to generalize  the proof of Theorem 13 in  \cite{GV}
to a slightly  more general setting. We recall that the proof in
\cite{GV} works only for $L(x,v)=K(v)-U(x)$, with suitable
semiconcavity/semiconvexity on $K$ and $U$.


Let  $u$ be a function with semiconcavity modulus smaller than  $\sigma$.
We will show that for a suitable  $\sigma$, the image $\Gg (u)$ has also modulus of concavity smaller than $\sigma$.
Because $\Gg$
commutes with constants, we can look at fixed points modulus constants. The set of functions with semiconcavity modulus bounded by $\sigma$ is
invariant by $\Gg$. When quotiented by the constants this set is compact and therefore $\Gg$ admits a fixed point modulo constants, which is precisely
the result of the theorem.

Consider
$$ u_1(x): = - \epsilon \, h \ln  \int e^{  - \frac{h  L(x,v) + u(x+h v) - \lambda_{\epsilon,h}
   }{\epsilon\,  h }} d\, v ,
$$
$$ u_1(x+ h \, y) = - \epsilon \, h
\ln  \int e^{  - \frac{h  L(x+ h\, y,v) + u(x+h\, y+h\, v) - \lambda_{\epsilon,h}
   }{\epsilon\,  h }} d\, v ,
$$
and
$$ u_1 (x -  h \, y) = - \epsilon \, h
\ln  \int e^{  - \frac{h  L(x- h\, y,v) + u(x-h\, y+h\, v) - \lambda_{\epsilon,h}
   }{\epsilon\,  h }        } d\, v .
$$

Let $0< \theta<1$, and $t=1-\theta$. Using the change of coordinates $v \to v -  \theta y$,
 we can write the second equation  as
$$ u_1(x+ h \, y) = - \epsilon \, h
\ln  \int e^{  - \frac{h  L(x+ h\, y,v-  \theta \, y) +
 u(x+h t y+h\, v) - \lambda_{\epsilon,h}
   }{\epsilon\,  h }} d\, v,
$$
whereas the third equation, through the change of coordinates
 $v \to v + \theta y$, can be written as
$$ u_1 (x-  h \, y) = - \epsilon \, h
\ln  \int e^{  - \frac{h  L(x- h\, y,v+   \theta \, y) +
 u(x- th\, y+h \, v) - \lambda_{\epsilon,h}
   }{\epsilon\,  h }} d\, v.
$$
Now using the hypothesis (3) of the Lagrangian  $L$, we get
$$ L( x+ h \, y, v - \theta \,y) + L( x- h \, y, v + \theta \,y) -
2\, L(x,v)
\leq C \, h^2 |y|^2 \, +\,
\Gamma \theta^2 |y|^2$$

We want to estimate the modulus of concavity of $u_1$ knowing that
$$u (x + h\,t\, y) + u (x-h \,t\,y) - 2 u (x) \leq
\sigma \, h^2\, t^2 |y|^2. $$
It is also true that
$$u (x + h\,t\, y+ h\, v) + u (x-h \,t\,y+ h\, v) - 2 u (x+ h\, v) \leq
\sigma \, h^2\, t^2 |y|^2. $$
Hence using the the concavity estimate of  $u$, we can write
$$ u_1(x)=  - \epsilon \, h \ln  \int e^{  - \frac{h  L(x,v) +
 u(x+h v) - \lambda_{\epsilon,h}
   }{\epsilon\,  h }} d\, v \geq$$
   $$- \epsilon \, h \ln  \int e^{  - \frac{[ hL( x+ h  y, v - \theta y)
     +
 u (x + ht y+ h v)  - \lambda_{\epsilon,h}]+[h L( x- h  y, v + \theta y)+ u (x-h ty+ h v)
   - \lambda_{\epsilon,h}]+[-
 C  h^3   -
h\Gamma \theta^2  -
\sigma_u  h^2 t^2 ]|y|^2
   }{2\epsilon\,  h }} d\, v
$$
$$  = - \epsilon \, h \ln  \int e^{  - \frac{
 [ \,\frac{1}{2}\, (h  L(x+ h y,v- \theta y) + u(x+th y+ h v) -
\lambda_{\epsilon,h}) \,+\, \frac{1}{2}
(h  L(x- h y,v+ \theta y) + u(x-thy+ h v) - \lambda_{\epsilon,h})\,]
 }{\epsilon\,  h }} d\, v-
$$
$$ -[\frac{\,C \, h}{2}+ \frac{\sigma_u\, t^2\, }{2}\,]\,\, h^2 \, |y|^2 -
\frac{\Gamma}{2} \, \theta^2 |y|^2 \, h.
$$
By Cauchy-Schwartz inequality we know that given functions   $a,b$  we have
 \[ \int a\,b\leq \left(\int a^2\right)^{\frac{1}{2}  }\, \left(\int b^2\right)^{ \frac{1}{2}  },\]
 hence using the expressions of   $u_1 (x+h y)$ and $u_1 (x- h y)$
we obtain
$$ u_1(x) \geq
\frac{1}{2} \,(u_1(x+h y) +  u_1  (x- h y) ) -[\frac{\,C \, h}{2}+
 \frac{\sigma_u\, t^2\, }{2}\,]\,\, h^2 \, |y|^2 -
\frac{\Gamma}{2} \, \theta^2 |y|^2 \, h.  $$ Therefore the semiconcavity modulus of  $u_1$ is  $\sigma_{u_1}=Ch+
 \sigma_u\, t^2 +
\Gamma \, \theta^2 / h. $

We want to choose a upper bound to the semiconcavity modulus of    $u$ such that the semiconcavity modulus of  $u_1$ is also smaller then this upper bound.  We claim that   $ \bar C=C+ \Gamma$ is the  bound which we are looking for. Indeed, suppose  $\sigma_u<\bar C$, by choosing $\theta=h$, and
taking
 $h$ small we have that
$$\sigma_{u_1}=Ch+
 \sigma\, t^2 +
\Gamma  h \leq(C+\Gamma)h+(C+\Gamma)(1-h)^2\leq C+\Gamma=\bar C$$
 Hence, as in theorem   26  of \cite{GV}, there exist a Lipschitz function $\phi_{\epsilon,h}$ and $\lambda_{\epsilon,h}\in\re$ such that $$\Gg[\phi_{\epsilon,h}]=\phi_{\epsilon,h}+\lambda_{\epsilon,h},$$ also the semiconcavity modulus   of $\phi_{\epsilon,h}$ is smaller than $\bar C$ for all $\ep$ and $h$.
\end{proof}

 \Rm It is easy to see that if we add a constant to each $\phi_{\epsilon,h}$ and $\bar\phi_{\epsilon,h}$, the equations \eqref{G} and \eqref{barG} are also satisfied. Then, for each $\ep$ and $ h$, we choose a pair of functions  $\phi_{\epsilon,h}$ and $\bar\phi_{\epsilon,h}$  and define a new pair of uniformly bounded functions   $\tilde \phi_{\epsilon,h}:=\phi_{\epsilon,h}-\phi_{\epsilon,h}(0)$ and $\tilde{\bar\phi}_{\epsilon,h}:=\bar\phi_{\epsilon,h}+c_{\ep,h}$ such that
\begin{equation}\label{theta}\int_{\tn} e^{-\frac{\tilde{\bar \phi}_{\epsilon,h}(x)+
      \tilde\phi_{\epsilon,h}(x)}{\epsilon h}}dx=1
      \end{equation}
As the functions $\tilde \phi_{\epsilon,h},\tilde{\bar\phi}_{\epsilon,h}$ are uniformly  Lipschitz in $\ep$ and $h$, we have that  $\tilde \phi_{\epsilon,h}$ is uniformly bounded. Moreover,  because $\tilde{\bar\phi}_{\epsilon,h}$ must satisfy the equation \eqref{theta}, we  get that $\tilde{\bar\phi}_{\epsilon,h}$ is also uniformly bounded in $\ep$ and $h$. Now on we will drop  the symbol $\;\;\tilde{}\;\;$.

\bigskip

We note that most of the results in \cite{GV} do not assume the Lagrangian is of the form $L(x,v)=K(v)-U(x)$. All  the results we need from \cite{GV} are true under the hypothesis (1) (2) (3) we mention above:

\begin{theorem} \label{teo2}
  Let $\phi_{\epsilon,h}$, $\bar \phi_{\epsilon,h}$ and
  $\lambda_{\epsilon,h}$ given by Theorem \ref{pfixo}. Also suppose that $\phi_{\epsilon,h}$ and $\bar \phi_{\epsilon,h}$ are uniformly bounded and satisfy \eqref{theta}.

  We define
  $\theta_{\epsilon,h}:\tn\to\re$ as
  \[
  \theta_{\epsilon,h}(x)=e^{-\frac{\bar \phi_{\epsilon,h}(x)+
      \phi_{\epsilon,h}(x)}{\epsilon h}}.
  \]

  Then, the probability density
$$\mu_{\epsilon,h}(x,v)=\theta_{\epsilon,h}(x)\;e^{-\frac{hL(x,v)+\phi_{\epsilon,h}(x+hv)-\phi_{\epsilon,h}(x)-\lambda_{\epsilon,h}}{\epsilon h}}$$
minimizes the functional
\[
 \int_{\tn\times\rn}
   L(x,v)d\mu(x,v)+\epsilon S[\mu]
\]
 over the densities in  $\mathcal M_h$.
\end {theorem}

\begin{proof}

Indeed,
$ \theta_{\epsilon,h}$
satisfies

$$
\int_{\rn}\theta_{\epsilon,h}(x-hv)e^{-\frac{hL(x-hv,v)+
\phi_{\epsilon,h}(x)-\phi_{\epsilon,h}(x-hv)-\lambda_{\epsilon,h}}{\epsilon
    h}}dv=
$$
$$\int_{\rn} e^{-\frac{\bar \phi_{\epsilon,h}(x-hv )+
      \phi_{\epsilon,h}(x-h v)}{\epsilon h}} \, e^{-\frac{hL(x-hv,v)+
\phi_{\epsilon,h}(x)-\phi_{\epsilon,h}(x-hv)-\lambda_{\epsilon,h}}{\epsilon h}}dv
=$$
$$
 e^{-\frac{ \phi_{\epsilon,h}(x)}{\epsilon h}}\,\int_{\rn} e^{-\frac{\bar \phi_{\epsilon,h}(x-hv )}{\epsilon h}} \, e^{-\frac{hL(x-hv,v)-\lambda_{\epsilon,h}}{\epsilon h}}dv=
$$
$$
  e^{-\frac{ \phi_{\epsilon,h}(x)}{\epsilon h}}\, \ e^{-\frac{\bar  \phi_{\epsilon,h}(x)}{\epsilon h}}=
\theta_{\epsilon,h}(x).
$$

Therefore, from Theorem 32 in \cite{GV} the result follows.

\end{proof}

\begin{theorem}\label{teo3}

Let $\phi_{\epsilon,h}$, $\bar \phi_{\epsilon,h}$ and
  $\lambda_{\epsilon,h}$ given by Theorem \ref{pfixo}. Also suppose that $\phi_{\epsilon,h}$ and $\bar \phi_{\epsilon,h}$ are uniformly bounded and satisfy \eqref{theta}.  Then, for $h$ fixed, when $\ep\to 0$, we have

(a)$$\overline H_{\epsilon,h}:=\int_{\tn\times\rn} L(x,v)d\mu_{\epsilon,h}(x,v)+\epsilon S[\mu_{\epsilon,h}]=\frac{\lambda_{\epsilon,h}}{h}
$$ and $\overline H_{\epsilon,h}\to\overline H_{h}$,

\medskip

(b) Through some subsequence,\,\,$\phi_{\epsilon,h}\to \phi_h,\,\,\bar\phi_{\epsilon,h}\to \bar\phi_h$  uniformly.   $\phi_h, \bar\phi_h$ are semiconcave functions, with the  semiconcavity constant bounded by $\bar C$ (as in  theorem \ref{pfixo}), and satisfy \begin{equation} \label{forsub}\phi_{h}(x)= \inf_{v \in \mathbb{R}^N } \, \{\phi_{h}(x+h\, v) + h\, L(x,v) -  h\, \Hh_h \}\end{equation} and \begin{equation} \label{bacsub}\bar\phi_h(x)= \inf_{v \in \mathbb{R}^N } \, \{ \bar\phi_h(x-h\, v) + h\, L(x-hv,v) -  h\, \Hh_h \}.
\end{equation}

\medskip

(c)\,\, $\mu_{\epsilon,h}\rightharpoonup \mu_h,$ where $\mu_h$ is a discrete Mather measure.
\end{theorem}

\begin{proof} From theorems 37 and 38 \cite{GV} and also by theorem \ref{teo2} we obtain item (a), by theorems 39 and 40 of \cite{GV} we get, respectively, (b) and (c).
\end{proof}

\vspace{.2cm}

If we use the so called Hopf-Cole transformation $\phi \to e^{- \frac{\phi}{\epsilon h} }\,=\, \varphi$, the setting above
can be written as the search for the eigenfunction associated to the largest
eigenvalue of the Perron operator $\varphi \to {\cal L} (\varphi)$
acting on continuous functions $\varphi$
$$ \,\, x \, \to \varphi(x) \, \Rightarrow \, x \to  \, {\cal L} \, (\varphi)\, (x) = \int \,
e^{-\, \frac{  L ( x,v)    }{\epsilon  }   }   \, \varphi(x+h\, v) \, d \, v.$$

The largest eigenvalue  of this operator is (see \cite{GV} Corollary 27) $ e^{-\, \frac{\lambda_{\epsilon, h}}{\epsilon h}}$.

\begin{definition} A property P is said to be generic for the Lagrangian $L$ if
there exists  a generic set $\mathcal O$ (in the Baire sense) on
the set $C^{\infty}(\tn,\re)$ such that if $\psi$ is in $\mathcal O$
then $L+\psi$ has property P.
\end{definition}

\begin{theorem} Given a Lagrangian  $L$ there exists a generic set  $\mathcal O\subset C^{\infty}(\tn)$ such that

(a) If  $\psi\in\mathcal O$ then there exists only one Mather measure for  $L+\psi$, such measure $\mu$ is uniquely ergodic.

(b) $\supp(\mu)=\hat{\mathcal A}(L+\psi)$, where $\hat{\mathcal A}$ is the Aubry set.

\end{theorem}
The proof of this theorem can be found in \cite{CP}.

\bigskip

\textbf{Assumption:} We will suppose that the Lagrangian $L(x,v) $ is generic, i.e., the Mather measure is unique, which we will denote by $\mu$, and $\supp(\mu)=\hat{\mathcal A}(L)$.

\bigskip

\Rm As we  suppose  the Lagrangian is generic, we have only  one static class, and the Mather measure is ergodic.  Then  by corollary 4-8.5 of \cite{CI} we know that the set of weak-KAM solutions (positive and negative) are unitary, modulo an additive constant. It can be shown, see \cite{Fa}, that  $-\phi$ is a positive weak-KAM solution, if and only if, $\phi$ is a viscosity solution of $H(\nabla\phi(x),x)=-\Hh_0$ (remember we are using the definition $H(p,x)=\sup_v(-p\cdot v -L(x,v)$), and $\bar \phi$ is a negative weak-KAM solution, if and only if, $\bar\phi$ is a viscosity solution of
$H(-\nabla\bar\phi(x),x)=-\Hh_0$.

Let us call $\phi_0$ and $\bar\phi_0$,  the unique viscosity solutions of $H(\nabla\phi(x),x)=-\Hh_0$ and $H(-\nabla\bar\phi(x),x)=-\Hh_0$, respectively.

Applying the corollary 5.3.7 of \cite{Fa} and the remark above, we obtain:
\begin{corollary}\label{bpeierls}Suppose that  the Lagrangian $L$ is generic, then  we have that
$$\phi_0(x)+\bar\phi_0(x)=h(x,x),$$ where $h$ is the Peierls barrier.
\end{corollary}

\begin{theorem}\label{teo5}Let $L(x,v)$ be a generic  Lagrangian  that satisfies the hypothesis (1) to (3) above. For each $h$, let $\phi_h, \bar\phi_h$ be  the functions,  $\mu_h$ be the measure,  and $\Hh_h$ be the constant that are given  in  theorem \ref{teo3}. Then, when $h\to 0$ we have

\medskip

 $(a)\,\,  \Hh_h\to\Hh_0=\int_{\tn\times\rn} L(x,v)d\mu, $

\medskip

  $(b)\,\,\mbox{Through some subsequence,\;\;\;} \phi_h\to\phi_0 \,\,\,\mbox{ and }\,\,\,\bar\phi_h\to\bar\phi_0,$ uniformly,

\medskip

$(c) \,\,\mu_h\rightharpoonup \mu.$

\end{theorem}
\begin{proof}
(a) See \cite{Gom}.

 In order to  apply theorems 7.2,7.3 and 7.4 of \cite{Gom} we need the following remark: as the Lagrangian satisfies the hypothesis (3) we have,  by item (b) of theorem \ref{teo3}, that $\phi_h$ and $\bar\phi_h$ are uniformly semiconcave in $h$. Let $\Lambda$ be the uniform Lipschitz constant. We claim that  each  $v(x)= v_h (x)$ that achieves the infimum in equation \eqref{forsub}  is uniformly bounded in $h$. Indeed,
$$|L(x,v(x))+\Hh_h|=|\frac{u(x)-u(x+hv)}{h}|\leq \Lambda |v(x)|,  $$ then, because the Lagrangian is superlinear and we have (a), we conclude that $|v(x)|\leq K$ for some constant $K$ that depends only on the Lagrangian $L$.

(b) Just note that $\phi_h$ and $\bar\phi_h$ are uniformly bounded, because they are limits of the functions $\phi_{\ep,h}$ that are uniformly bounded in $\ep$ and $h$, hence we can apply theorem 7.2 of \cite{Gom}.

(c) See theorems 7.3 and 7.4 of \cite{Gom}.
\end{proof}

\begin{theorem}\label{viscosity}Let $L(x,v)$ be a generic  Lagrangian  that satisfies hypothesis (1) to (3) above. Suppose that $\phi_{\epsilon,h}$ and $\bar \phi_{\epsilon,h}$ given by Theorem \ref{pfixo}  are uniformly bounded and satisfy \eqref{theta}. Then, through some subsequence, $$\phi_{\epsilon, h} \to \phi_0 \mbox{ \;\;\;\;and\;\;\;\; } \bar\phi_{\epsilon, h}\to\bar \phi_0.$$

\end{theorem}
\begin{proof}

By item (b) of theorem \ref{teo5},   we know that  any collection $\overline{\{\phi_{ h}\}}_{h\in[0,1]}$ of solutions of the   $\epsilon=0$ problem  is a compact set, then
if we take a sequence $\{\phi_{ h_i}\}_{i\in\Nn}$ it has a subsequence that converges to  $\phi_0$, i.e., there exists a set $\mathcal H$ such that $$\lim_{h_i\in\mathcal H}\phi_{ h_i}=\phi_0.$$
 We know by theorem 39 of  \cite{GV}, for each $h_i\in\mathcal H$ fixed (as $\phi_{\ep,h_i}$ are normalized), that   $\overline{\{\phi_{\epsilon, h_i}\}}_{\epsilon\in[0,1]}$ is a compact set.
Then if we fix $h_1\in\mathcal H$ and a sequence $\{\phi_{\epsilon_i, h_1}\}_{i\in\Nn}$, then there exists a set $\mathcal E_{h_1}$ such that
$$\lim_{\epsilon_i\in\mathcal E_{h_1}}\phi_{\epsilon_i, h_1}=\phi_{h_1}.$$

Then, if we do this for each $h_i\in\mathcal H$,  we can find a set $\mathcal E_{h_i}\subset ...\subset\mathcal E_{h_2}\subset\mathcal E_{h_1}$.
Now we define a set $\mathcal E$ such that the i-th element of $\mathcal E$ is the i-th element of $\mathcal E_i$. The set $\mathcal E$ has the property that $$\lim_{\epsilon_i\in\mathcal E}\phi_{\epsilon_i, h_j}=\phi_{h_j} \mbox{ for each } h_j\in\mathcal H.$$
Finally, we have that
$$\lim_{i\to\infty}\phi_{\epsilon_i, h_i}=\phi_0.$$\end{proof}

\section{A large deviation principle:  $h$ fixed and $\epsilon\to 0$}

\begin{lemma}[Laplace Method] \label{laplace}

If  $f_k(x,v)\to f_0(x,v)$ uniformly  as $k\to 0$, then
for each $A\subset\tn\times\rn$  closed bounded set, we have $$\limsup_{k\to 0}\; k\ln\int_A e^{-\frac{f_k(x,v)}{k}}dx dv\leq -\inf_{A}f_0(x,v),$$
and for each $A\subset\tn\times\rn$ open bounded set, we have
$$\liminf_{k\to 0}\; k\ln\int_A e^{-\frac{f_k(x,v)}{k}}dx dv\geq - \inf_{A}f_0(x,v).$$

\end{lemma}

\bigskip

Let us define,

\[
f_{\epsilon,h}(x,v)=\frac{
\bar \phi_{\epsilon, h}(x)+\phi_{\epsilon, h}(x)}{h}
+
L(x,v)+\frac{\phi_{\epsilon,h}(x+hv)-\phi_{\epsilon,h}(x)}h-\overline{H}_{\epsilon,h},
\]
and
\[
I_h (x,v)=\frac{\bar\phi_h(x)+\phi_h(x+hv)}{h}+
L(x,v)-\overline{H}_{h},\mbox{\;\;\;\;\;  for any $ (x,v)  . $ }
\]

In order to have $I_h$ defined in a unique way we need the uniqueness
of  $\phi_h$ and $\bar\phi_h$.  In the last section we will show a sufficient condition to that.

\begin{theorem}\label{ldph}

Consider $A\subset \tn\times\rn$ a closed (resp. open ) bounded set, then
$$\lim_{\epsilon \to 0} \epsilon \ln  \mu_{\epsilon,h}(A) =
\lim_{\epsilon \to 0} \epsilon \ln  \int_A
e^{-\, \frac{f_{\epsilon,h}(x,v)}{\epsilon } } dx dv\leq -\, \inf_{(x,v ) \in A } I_h (x,v)\,\, (\mbox{ resp. } \geq).  $$

\end {theorem}
\begin{proof}

As for $h$ fixed, the convergence of $  \phi_{\epsilon,h} $ and $\bar  \phi_{\epsilon,h}$, with $\epsilon \to 0$,  to respectively,
$  \phi_{h} $ and $ \bar  \phi_{h} ,$ is uniform by item (b) of theorem \ref{teo3}. Then, the proof follows from the lemma \ref{laplace}  (Laplace method).
\end{proof}

\section{A large deviation principle: $\epsilon, h \to 0$}








Thanks to  \cite{FS} we can assume the Lagrangian $L$ we consider here satisfies the property that $I(x,v)=0$, if and
only if, $(x,v)$ is in the support of the Mather measure $\mu$.

Note that by theorem \ref{viscosity} there exists a  sequence $\{\epsilon_i,h_i\}_{i\in\nat}$ such that $\epsilon_i,h_i\to 0 $ and  $\displaystyle\lim_{i\to
  \infty} \phi_{\epsilon_i,h_i}=\phi_0$, and $\displaystyle\lim_{i\to\infty} \bar \phi_{\epsilon_i,h_i}=\bar \phi_0$.
For convenience we will write $\displaystyle\lim_{\epsilon,h\rightarrow0}$ when we
want to mean $\displaystyle\lim_{\epsilon_i,h_i\rightarrow0}$.

All the results that we will obtain  will be independent of the
particular sequence we choose, because $\phi_0$ and $\bar\phi_0$ are uniquely determined.

\begin{theorem}
\label{auxlem}
 If $x\in\mbox{dom}(\nabla \phi_0)$, then we have
$$\lim_{\epsilon,h\rightarrow0}\frac{\phi_{\epsilon,h}(x+hv)-\phi_{\epsilon,h}(x)}{h}=\nabla\phi_0(x)(v),$$
uniformly in each closed bounded subset of
$\mbox{dom}(\nabla \phi_0)\times \rn$.
\end{theorem}


To prove theorem \ref{auxlem} we need the following  properties of semiconcave functions
(see[CS, Ch.3]).
\begin{proposition}\label{clarke} Let $u:\tn\to\re$ be  a semiconcave function. Given $x,y\in\tn$ there exist $\xi\in ]x,y[$ and
  $p\in D^+u(\xi)$ such that $u(y)-u(x)=p\cdot (y-x)$, where $D^+u(x)$ is the
  superdifferential of $u$ at $x$.
\end{proposition}


\begin{proposition} \label{desigsuperdif} Let $u:\tn\to\re$ be a
  semiconcave function with semiconcavity modulus $C$, and let
  $x\in\tn$. Then, a vector $p\in\rn$ belongs to $D^+u(x)$ if and only
  if
$$ u(y)-u(x)\leq p\cdot(y-x)+\frac{C}{2}|y-x|^2  $$ for any point $y\in\tn$.

\end{proposition}

\begin{proof}
{{\textit{(of Theorem \ref{auxlem})}}} By Theorem \ref{pfixo},
the functions
$\phi_{\epsilon,h}$
are  semiconcave with semiconcavity modulus uniformly bounded
by some constant $\bar C$. Let $\{\epsilon_i,h_i\}_{i\in \mathbb{N}}$ be a
sequence  such that
$\phi_{\epsilon_i,h_i}\to \phi_0$.

Let $K$ be a closed bounded subset of $\mbox{dom}(\nabla \phi_0)\times\rn$.  Hence, by propositions
$\ref{clarke}$ and $\ref{desigsuperdif}$, for each $(x,v)\in K$, and each
$\epsilon_i$ and $h_i$ there exist
$\xi_{\epsilon_i,h_i}(x,v)\in]x,x+h_iv[$\; and $p_{\epsilon_i,h_i}\in
D^+\phi_{\epsilon_i,h_i}(\xi_{\epsilon_i,h_i}(x,v))$, such that
$$\frac{\phi_{\epsilon_i,h_i}(x+h_iv)-\phi_{\epsilon_i,h_i}(x)}{h_i}
=p_{\epsilon_i,h_i}\cdot
v.$$ Then, in order to prove the lemma it is enough show that
$$\lim_{\epsilon_i,h_i\rightarrow0}p_{\epsilon_i,h_i}\cdot
v=\nabla\phi_0(x)(v)\;\;\;\;\; \mbox{ for all } (x,v)\in K,$$ i.e.,
given $\zeta>0$ we need to find $i_0\in\nat$ such that for each $i>i_0$ and $(x,v)\in K$ we have
\begin{itemize}
\item[(i)]$\nabla\phi_0(x)(v)\leq p_{\epsilon_i,h_i}\cdot v+\zeta$
\item[(ii)]$ \nabla\phi_0(x)(-v)\leq -p_{\epsilon_i,h_i}\cdot v+\zeta$.
\end{itemize}

Firstly, we will show that there exists $i_0$, such that the first
inequality holds for every $i>i_0$, and every $(x,v)\in K$.
Arguing by contradiction, we suppose that there is no $i_0>0$,
with the specified properties. Then there exists a sequence
$\{(x_n,v_n)\}$, and subsequences $ \{h_n\}, \{\epsilon_n\}$ of $
\{h_i\}, \{\epsilon_i\}$, such that
\begin{equation}\label{desig1} \nabla\phi_0(x_n)(v_n)>
  p_{\epsilon_n,h_n}\cdot v_n+\zeta,
\end{equation}
where $p_{\epsilon_n,h_n}\in D^+\phi_{\epsilon,h}(\xi_n)$ and
$\xi_n:=\xi_{\epsilon_n,h_n}(x_n,v_n)\in]x_n,x_n+h_nv_n[$ are given by
Proposition \ref{clarke}.  Passing to a subsequence, if necessary, we
can suppose that the sequence $\{(x_n,v_n)\}$ converges to a point
$(x,v)$ of K, then $\{\xi_n\}$ converges to $x$.  Now, by Proposition
\ref{desigsuperdif}, we have that, for any $\lambda>0$
\begin{equation}\label{desig2} p_{\epsilon_n,h_n}\cdot v_n \geq
  \frac{\phi_{\epsilon_n,h_n}(\xi_n+\lambda
    v_n)-\phi_{\epsilon_n,h_n}(\xi_n)}{\lambda}
  -\frac{\bar C}{2}\lambda|v_n|^2
\end{equation}

Note that $\phi_{\epsilon_n,h_n}\to\phi_0$ when $n\to\infty$
uniformly, $\xi_n\to x$ and $\nabla\phi_0$ is continuous in $\mbox{dom}(\nabla \phi_0)$.
Then
by equations (\ref{desig1}) and (\ref{desig2}), we have that
\begin{align*}
&\lim_n\nabla\phi_0(x_n)(v_n)\geq \liminf_n p_{\epsilon_n,h_n}\cdot v_n+\zeta\\
&\qquad\geq \lim_n \frac{\phi_{\epsilon_n,h_n}(\xi_n+\lambda v_n)-\phi_{\epsilon_n,h_n}(\xi_n)}{\lambda} -\frac{\bar C}{2}\lambda|v_n|^2 +\zeta,
\end{align*}hence
$$\nabla\phi_0(x)(v)\geq   \frac{\phi_0(x+\lambda v)-\phi_0(x)}{\lambda} -\frac{\bar C}{2}\lambda|v|^2 +\zeta,\;\;\;\; \mbox{ for all } \lambda>0.$$
Then
$$\nabla\phi_0(x)(v)\geq \lim_{\lambda\downarrow 0}  \frac{\phi_0(x+\lambda v)-\phi_0(x)}{\lambda} -\frac{\bar C}{2}\lambda|v|^2 +\zeta=\nabla\phi_0(x)(v)+\zeta,$$ and this is a contradiction. Repeating the argument with $v$ replaced
by $-v$ yields the other inequality.
\end{proof}

\begin{theorem} \label{supp}
Consider $\phi_0$ and $\bar \phi_0$  the functions given by theorem \ref{viscosity} and denote by $\mu$ the Mather measure for $L$.
Then,
$$
\pi_1(\supp(\mu)) =  \{x\,:\, \phi_0(x) +
\bar \phi_0  (x) =0\},$$ where $\pi_1$ is the canonical projection on
the $x$ coordinate.
\end{theorem}
\begin{proof}
This follows by the corollary \ref{bpeierls} (as the Lagrangian is generic), because
 the Peierls barrier $h(x,x)=0$, if and only if, $x$ is in the
projection of the support of the Mather measure (the projected Aubry set).
\end{proof}

\begin{theorem}\label{limite}
Let us fix two  sequences $\{\ep_n\}, \{h_n\} $ such that $h_n\geq \epsilon_n $, $\mu_{\epsilon_n, h_n }\rightharpoonup \mu, \overline{H}_{\ep_n,h_n}\to\Hh_0, \phi_{\epsilon_n, h_n}\to\phi_0  $ and $\bar\phi_{\epsilon_n, h_n}\to\bar\phi_0  $ . To simplify the notation we will denote by $\mu_n=\mu_{\epsilon_n, h_n },\,\overline{H}_n=\overline{H}_{\ep_n,h_n} $, $\phi_n=\phi_{\epsilon_n, h_n} $ and $\bar\phi_n=\bar\phi_{\epsilon_n, h_n} $. Then, we have that

 (a)\; $\displaystyle\liminf_{n\to\infty}\frac{\bar \phi_n(x)+\phi_n(x)}{h_n}\geq 0,  \;\;\; \;\;\;\forall x\in\tn    ,$

 \bigskip

(b) \;  $\displaystyle\lim_{n\to\infty} \frac{\bar \phi_{n}  (x) +\phi_{n}  (x) }{h_n\,} =\infty, \;\;\;\mbox{ if }\;\;\;x\notin \pi_1(\supp(\mu)),$

 \bigskip

(c)  \;   $\displaystyle\limsup_{n\to\infty}\; \inf_{x\in B_{x_0}(r)}\frac{\bar \phi_{n}  (x_0) +\phi_{n}  (x_0) }{h_n\,} =0,\;\;\;  \mbox{ if }\;\;\; x_0\in\pi_1(\supp(\mu))$, for all $r>0$.
\end{theorem}

\begin{proof} (a)   Suppose  by contradiction
  that  there exists $x\in \tn$ such that for a subsequence $\displaystyle\lim_{j\to\infty}\frac{\bar \phi_{n_j}(x)+\phi_{n_j}(x)}{h_{n_j}}=c< 0$, then there exists a neighborhood $V$ of  $x$ of diameter $\bar c h_{n_j}$, where $\bar c$ is a constant,  such that  $\frac{\bar \phi_{n_j}(y)+\phi_{n_j}(y)}{h_{n_j}}\leq c/2$ for all $y\in V$. Then

$$\displaystyle e^{-\frac{c}{2\;\epsilon_{n_j}}}\int_{V}dx\leq \int_{V} e^{-\frac{1}{\epsilon_{n_j}}\frac{
\bar \phi_{n_j}(x)+\phi_{n_j}(x)}{h_{n_j}}} \, dx<\displaystyle\int_{\tn} e^{-\frac{1}{\epsilon_n}\frac{
\bar \phi_n(x)+\phi_n(x)}{h_n}} \, dx\,=\, 1.$$ But, $e^{-\frac{c}{2\;\epsilon_{n_j}}} (\bar c h_{n_j})^N\to\infty$ when $n_j\to \infty $, then  we get a contradiction, as $c<0$.

\medskip

(b) It follows by item (a) and theorem \ref{supp}.

\medskip

(c) First, we fix a  point  $(x_0,v_0)$ in the support of $\mu$ and let $B$ be a small
neighborhood of $(x_0,v_0)$ in the phase space.  As $\mu_n\rightharpoonup \mu $ there exists $n_0\in\nat$ such that if $n\geq n_0$ then
\begin{equation}\label{desig}1> \int_B e^{-\frac{1}{\epsilon_n}( \,       \frac{
\bar \phi_n  (x) +\phi_n  (x) }{h_n\,
 }\,\,
+\,\,\frac{
h_n\, L(x,v)+ \phi_n(x+h_nv)-\phi_n(x) -h_n\,\overline{H}_n}{
 \, h_n        }\,)\, }   \, dx dv\, \,>\delta_B>0\end{equation}
for some positive $\delta_B$.

\medskip

\textbf{{Claim:}} Given $\zeta>0$ there exists $\bar n\in \nat$ and a neighborhood $B$ of $(x_0,v_0)$ such that, if $(x,v)\in B$ and $n>\bar n $ then $\displaystyle L(x,v)+ \frac{\phi_n(x+h_nv)-\phi_n(x)}{ h_n} -\overline{H}_n  >-\zeta$.

\medskip

 We postpone the proof of the claim. Suppose by contradiction that
 \[\limsup_{n\to\infty}\;\inf_{x\in B_{x_0}(r)}\frac{\bar \phi_n(x)+\phi_n(x)}{h_n}=c> 0,
 \] then there exists a subsequence such that $\displaystyle\lim_{j\to\infty}\;\inf_{x\in B_{x_0}(r)}\frac{\bar \phi_{n_j}(x)+\phi_{n_j}(x)}{h_{n_j}}=c$. Then, there exists $j_0$ such that for  $j>j_0$ we have \begin{equation}\label{bola}\frac{\bar \phi_{n_j}(x)+\phi_{n_j}(x)}{h_{n_j}}>\frac{c}{2}\;\;\;\;\;\;\;\forall x\in B_{x_0}(r)\end{equation} Let $\tilde B=B_{x_0}(r)\times \rn.$
 Now using the claim with $\zeta<c/4$,  let $B$   be the neighborhood in the claim. Take $\hat B= B\cap \tilde B$, jointing the inequalities \eqref{bola} and that of the claim   we have a contradiction, when $\ep_{n_j}\to 0$, with the inequality (\ref{desig}). This  proves (c).

 \medskip

 \textit{Proof of the claim:} Let $\bar C$ be the semiconcavity bound of the functions $\phi_{\ep,h}$. For $\zeta>0, \eta>0$ there exists $\lambda>0$ such that

  $\frac{\bar C}2(|v_0|+\eta)^2\lambda<\zeta $ \,\,and\,\,
 $\displaystyle\left|\frac{\phi_0(x_0+\lambda v_0)-\phi_0(x_0)}{\lambda}-\nabla\phi_0(x_0)(v_0)\right|<\zeta$.

  As $\phi_n\to\phi_0$ uniformly in $x$, there exists $n_2$ such that if $n>n_2$, then we have $|\phi_0(x)-\phi_n(x)|<\zeta \lambda$, for all $x\in\tn$. Also there exists a neighborhood $B_{\lambda}$ of $(x_0,v_0)$ such that, if $(x,v)\in B_{\lambda}$ and $n>n_2$, then $$\left|\frac{\phi_n(x+\lambda v) -\phi_n(x)-\phi_0(x_0+\lambda v_0)+ \phi_0(x_0)}{\lambda}\right|<6\zeta. $$  There exists $n_3$ such that if $n>n_3$ and $(x,v)\in B_{\lambda}$ (choosing  $B_{\lambda}$ smaller if necessary) such that $|L(x_0,v_0)-L(x,v)-\Hh_0+\Hh_n|<\zeta$ and $|v-v_0|<\eta$.

By propositions \ref{clarke} and \ref{desigsuperdif} we have that
$$\frac{\phi_n(x+h_nv)-\phi_n(x)}{ h_n}=p_n(x_n)\cdot v\geq \frac{\phi_n(x_n+\lambda v)-\phi_n(x_n)}{ \lambda}-\frac {\bar C}2\lambda |v|^2,$$ where $x_n\in]x,x+h_n v[$, therefore there exists $n_4$ such that if $n>n_4$ then $(x_n,v)\in B_{\lambda}$.

Now, define $\bar n=\max\{n_2,n_3, n_4\}$,  collecting all the above inequalities, for any $n>\bar n$ and $(x,v)\in B_{\lambda}$,  we get
$$ \displaystyle L(x,v)+ \frac{\phi_n(x+h_nv)-\phi_n(x)}{ h_n} -\overline{H}_n  > L(x_0,v_0)+\nabla\phi_0(x_0)(v_0)-\Hh_0-9\zeta>-9\zeta,$$ which proves the claim.\end{proof}

Let us define the deviation function $I$ by
$$I(x,v)=L(x,v)+\nabla\phi_0(x)(v)-\overline{H}_{0}.$$

We remember the reader  that we just consider $I(x,v)$ for the points $x$ where $\nabla\phi_0(x)$ is defined. For the others points $x$ we declare $I(x,v)=\infty$.

\begin{proposition} Let $\phi_0$ be a   viscosity solution to
$
H(\nabla\phi_0(x),x)=-\overline H_0.
$

If $(x,v)\in \supp (\mu)$ then
  $\nabla\phi_0(x)(v)+L(x,v)=\overline H_0$.
\end{proposition}
\begin{proof}By theorem 4.8.3 of \cite{Fa} we have that $\phi_0$ is differentiable in $\pi_1(\supp (\mu))$. Let $(x,v)\in \supp (\mu)$, by corollary 4.4.13 of \cite{Fa} we obtain $H(\nabla\phi_0(x),x)=-\overline H_0.$ Therefore \begin {equation}\label{5}\nabla\phi_0(x)(v)+ L(x,v)\geq \overline H_0.
\end{equation}  To get the other inequality, suppose, by
contradiction, that there exists $(x,v)\in \supp(\mu)$ and
$\epsilon>0$ such that
\[
\nabla\phi_0(x)(v)+ L(x,v)> \overline
H_0+\epsilon.
\]
Then there is a neighborhood $V$ of $(x,v)$ such that
for all $(v,w)\in V $ we have
\[\nabla\phi_0(y)(w)+ L(y,w)> \overline
H_0.
\]
We recall that $\int\nabla\phi_0(x,v)d\mu=0$, then
\[
\int\nabla\phi_0(x)(v)+L(x,v)d\mu>\overline H_0,
\]
because (\ref{5})
is true at any point $(x,v)\in\pi_1(\supp\mu)\times \rn$ and at the
points $(x,v)\in V$ we have the strict inequality.

This implies
\[
\int L(x,v)d\mu>\overline H_0,
\]
but this is a
contradiction.

\end{proof}
If we fix $x$,  we have that
\begin{equation}\label{legendre}-\inf_v
I(x,v)=\sup_v(-I(x,v))=H(\nabla\phi_0(x),x)+\Hh_0.  \end{equation}
 We
know that if $x\in\pi_1(\supp(\mu))  $ then
$H(\nabla\phi_0(x),x)+\Hh_0 =0$, and by the hypothesis that the
Lagrangian is strictly convex in $v$, we obtain that there is just one
$v$  which achieves the supremum  in \eqref{legendre}. Moreover,  as we know
that  $(x,v)\in \supp(\mu)$, if and only if, $ I(x,v)=0$, we conclude that
$I(x,v)>0$, for all $(x,v)\notin \supp(\mu)$ with
$x\in\pi_1(\supp(\mu)).$

It makes sense to look for lower and upper deviations inequalities
just  in the case
$\displaystyle\inf_A I(x,v)>0$.

\begin{theorem}
Let us denote $D=\mbox{dom}(\nabla \phi_0)$.
Let  $A\subset D\times \mathbb{ R}^N   $ be  such that
$D\cap \pi_1(\supp(\mu))\neq \emptyset$, but $d(A ,\supp(\mu))\geq c>0$. Then

(a) if $A$ is a closed bounded set in $D\times \mathbb{ R}^N$ we have
$$
\lim_{\epsilon,h\rightarrow0}\epsilon\,  \ln
  \mu_{\epsilon,h}(A)\,\leq \, -\, \inf_{(x,v) \in A } I(x,v),$$

  (b) if $A$ is an open bounded set in $D\times \mathbb{ R}^N$
  $$\lim_{\epsilon,h\rightarrow0}\epsilon\,  \ln
  \mu_{\epsilon,h}(A)\,\geq \, -\, \inf_{(x,v) \in A_1 } I(x,v),$$
  where $A_1=\{(x,v)\in A\;:\;x\in \pi_1 (\supp(\mu))\}$.
\end{theorem}

{\bf Remark on item (b):} Given a set $A$ as above, consider $B^\delta=\{(x,v) \in A$, such that $d(x, \supp(\mu))\geq \delta>0\}$, for any fixed small $\delta$. From theorem \ref{limite} (and theorem \ref{auxlem}) we have that
$$\lim_{\epsilon,h\rightarrow0}\epsilon\,  \ln
  \mu_{\epsilon,h}(A)\,= \lim_{\epsilon,h\rightarrow0}\epsilon\,  \ln
  \mu_{\epsilon,h}(B^\delta)\,.$$

In this way, the lower bound  $-\, \inf_{(x,v) \in A_1 } I(x,v)$ is the precise information that makes sense.
In other words, the values $I(x,v)$ outside $A_1$ are not relevant.
\begin{proof}

(a) Note that  $A\subset D\times B_R$,
where $B_R=\{v\in \rn;|v|\leq R\}$, for some  $R>0$.

Remember that

\[
\mu_{\epsilon,h}(A)= \int_A e^{-( \,       \frac{ \bar
\phi_{\epsilon, h}  (x) +\phi_{\epsilon, h}  (x) }{ \epsilon \,h
 }\,\,
+\,\,\frac{ h\, L(x,v)+ \phi_{\epsilon,h}(x+hv)-\phi_{\epsilon,h}(x)
-h\,\overline{H}_{\epsilon,h}}{ \epsilon \, h        }\,)\, }dx dv,
\]

then
\begin{align*}
\mu_{\epsilon,h}(A)&\leq \,  e^{-\inf_{A}{( \,       \frac{ \bar
\phi_{\epsilon, h}  (x) +\phi_{\epsilon, h}  (x) }{\epsilon \, h
 }\,\,
+\,\,\frac{ h\, L(x,v)+ \phi_{\epsilon,h}(x+hv)-\phi_{\epsilon,h}(x)
-h\,\overline{H}_{\epsilon,h}}{ \epsilon \, h        }\,)\,
}}\int_{A} dx\,\, dv
\end{align*}

Hence
 $$\epsilon\ln\mu_{\epsilon,h}(A)\leq$$

$$\leq -\inf_{A}{\left(\frac{\bar \phi_{\epsilon, h}(x) +\phi_{\epsilon, h}(x) }{h}\right) -\inf_{A} \left({
L(x,v)+\frac{
   \phi_{\epsilon,h}(x+hv)-\phi_{\epsilon,h}(x)}{h} -\overline{H}_{\epsilon,h}}\right) }\,+\,\epsilon\ln c_1R^N .             $$

By item (c) of theorem  \ref{limite}
we have
$$\limsup_{\epsilon,h\to 0}\inf_{x\in A}   \frac{ \{\bar \phi_{\epsilon, h}(x) +\phi_{\epsilon, h}(x) \}}{h}=0$$

This implies that

$$\limsup_{\epsilon,h\to 0}\;\epsilon\ln\mu_{\epsilon,h}(A)\leq -\inf_{x\in A} I(x,v).$$

(b) Let $A\subset D\times\rn$ be a bounded and open set in $D\times\rn$, such that $A\cap A_1\neq\emptyset$. We fix  $\delta>0$, as $I(x,v)$ is continuous in $D\times\rn$ (see theorem 4.9.2 of \cite{Fa}), we can find an open set $A_{\delta}$ such that: $\bar{A_{\delta}}$ is a closed set in $D\times\rn$, $\bar {A_{\delta}}\cap A_1\neq\emptyset$   and

 $$
I(x,v)\leq\inf_{A_1}I(x,v)+\delta\;\;\;\mbox{ for all }\;\;\;(x,v)\in \bar{A_{\delta}}$$

Therefore  $$ \mu_{\epsilon,h}(A)\geq
\mu_{\epsilon,h}(\bar{A_{\delta}})\geq
e^{-\sup_{\bar{A_{\delta}}}\left(\frac{ h\, L(x,v)+ \phi_{\epsilon,h}(x+hv)-\phi_{\epsilon,h}(x)
-h\,\overline{H}_{\epsilon,h}}{ \epsilon \, h        }\right)}\int_{\bar{A_{\delta}}}e^{-( \,       \frac{ \bar
\phi_{\epsilon, h}  (x) +\phi_{\epsilon, h}  (x) }{ \epsilon\,h
 })} dx dv  $$

As $\bar{A_{\delta}}\cap A_1\neq \emptyset$ and $A_{\delta}$ is an open set there exists $c_{\delta}>0$ such that
$$1\geq \liminf_{\epsilon,h\to 0}\int_{\bar{A_{\delta}}}e^{-( \,       \frac{ \bar
\phi_{\epsilon, h}  (x) +\phi_{\epsilon, h}  (x) }{\epsilon\,h
 })} dx dv  \geq c_{\delta}.$$

We have that
$$ \liminf_{\epsilon,h\to 0}\; \epsilon\ln \mu_{\epsilon,h}(A)\geq -\sup_{\bar{A_{\delta}}}I(x,v) \geq -\inf_{A_1} I(x,v)+\delta $$
Making  $\delta\to 0$ we obtain
$$ \liminf_{\epsilon,h\to 0}\; \epsilon\ln \mu_{\epsilon,h}(A)\geq -\inf_{A_1} I(x,v).$$

\end{proof}

\begin{theorem}\label{ldph}
If $A\subset D\times\rn$ is such that
$D\cap \pi_1(\supp(\mu))= \emptyset$.
If $A$ is  closed and  bounded we have
$$
\lim_{\epsilon,h\rightarrow0}\epsilon\, h \ln
  \mu_{\epsilon,h}(A)\leq\,-\,   \inf_{x\in D}    \{\bar \phi_{0}(x) +\phi_{0}(x) \}.
$$
And if $A$ is  open and  bounded we have
$$
\lim_{\epsilon,h\rightarrow0}\epsilon\, h \ln
  \mu_{\epsilon,h}(A)\geq \,-\,   \inf_{x\in D}    \{\bar \phi_{0}(x) +\phi_{0}(x) \}.
$$
\end{theorem}
\begin{proof}

We can write  $$\mu_{\ep,h}(A)=\int_A e^{-\frac{\tilde f_{\epsilon,h}(x,v)}{\ep h}}dxdv,$$ where
 $$
\tilde f_{\epsilon,h}(x,v)=\bar \phi_{\epsilon, h}(x)+\phi_{\epsilon, h}(x)+h L(x,v)+\phi_{\epsilon,h}(x+hv)-\phi_{\epsilon,h}(x)-h\overline{H}_{\epsilon,h}.
$$
As $\tilde f_{\epsilon,h}(x,v)\to \bar \phi_{0}(x) +\phi_{0}(x)$ uniformly, using lemma \ref{laplace} (Laplace Method) we get the two inequalities of the theorem.
\end{proof}

We have some final comments about the large deviation problem. For a
fixed $(x,p)$ consider
$$ Z_{\epsilon,h,p} (x)= \int \, e^{\frac{\langle p,v\rangle}{\epsilon } \,}
\, d \, \mu_{\epsilon, h}(x,v) \,\,\,$$
and the free energy $$c(p, x)= \lim_{\epsilon, h \, \to 0}\,
\epsilon\, \ln Z_{\epsilon,h,p} (x),$$
where $\mu_{\epsilon, h}$ was chosen for $L$ as above.

\begin{theorem}
For each, $(x,p)$, for almost everywhere (Lebesgue) $x$
$$ c(p,x) = H(\,  \nabla \phi_0(x)-p,x \,)+\Hh_0 $$
\end{theorem}
\begin{proof}
As
$$\mu_{\epsilon,h}(x,v)=\theta_{\epsilon,h}(x) \,\gamma_{\epsilon,h}(x,v)\,= \, e^{-\, \frac{\bar  \phi_{\epsilon,h}(x)+ \phi_{\epsilon,h}(x)   }{\epsilon \, h }}\,  e^{-\frac{hL(x,v)+\phi_{\epsilon,h}(x+hv)-\phi_{\epsilon,h}(x)-\lambda_{\epsilon,h}}{\epsilon h}},$$
then
$$Z_{\epsilon,h,p} (x)= \int \, e^{\frac{\langle p,v\rangle}{\epsilon } \,}\,  \,\theta_{\epsilon,h}(x) \,\gamma_{\epsilon,h}(x,v)\, dx\, dv.$$
As $H$ is the Legendre transform of $L$, the result follows from
the results we obtained before.
\end{proof}

Therefore, the Legendre transform of the free energy is the
deviation function.

\vspace{0.3cm}

\noindent {\bf Example.} An interesting example is the following:

Consider $L(x,v)=\frac{v^2}{2} $.

Then
$$\Gg[\phi](x):= -\epsilon h \mbox{ ln}\left[
  \int_{\re^N} e ^{-\frac{h\,       \frac{v^2}{2}  +\phi(x+hv)}{\epsilon h}}dv\right] ,$$
satisfies
$$\Gg[0](x)  = -\epsilon h \mbox{ ln} (
  \int_{\re^N} e ^{-\frac{        h\, \frac{v^2}{2} }{\epsilon\,h }}dv) =  -\epsilon h \mbox{ ln} \, \sqrt{2 \pi \epsilon }  + 0  =
\lambda_{\epsilon,h}        .$$
Therefore $\theta_{\epsilon, h}=1$ and
$$\mu_{\epsilon,h}(x,v)=
e^{-       \frac{   \,  \frac{v^2}{2}       -\Hh_{\epsilon,h}  }{\epsilon }}.$$

In this case,

$$  S[\mu_{\epsilon,h} ]=\int_{\tn\times\rn} \mu_{\epsilon,h}(x,v)\ln\frac{\mu_{\epsilon,h}(x,v)}{\int_{\rn}\mu_{\epsilon,h}(x,w)dw}dxdv=
$$
$$\int_{\rn}   \,   e^{-       \frac{   h\,  \frac{v^2}{2}       -\lambda_{\epsilon,h}  }{\epsilon \, h}  }
\,  \ln \,(\,        e^{-       \frac{   h\,  \frac{v^2}{2}       -\lambda_{\epsilon,h}  }{\epsilon \, h}}        \,)\, dv=
$$
$$\int_{\rn}   \,   e^{-       \frac{   h\,  \frac{v^2}{2}       -\lambda_{\epsilon,h}  }{\epsilon \, h}  }
\,  (-       \frac{   h\,  \frac{v^2}{2}       -\lambda_{\epsilon,h}  }{\epsilon \, h} )\, dv\,=
$$
$$ =-\mbox{ ln} \, \sqrt{2 \pi \epsilon } -{\frac{1  }{ \sqrt{2 \pi \epsilon } }}\int \frac{v^2}{2\epsilon}\; e^{-\frac{v^2}{2\epsilon}   }dv=-\mbox{ ln} \, \sqrt{2 \pi \epsilon }-\frac{1}{2}$$

Therefore, the term $\epsilon\, S(\mu_{\epsilon,h})$ goes to $0$ when
$\epsilon \to 0$. We point out that  $S(\mu_{\epsilon,h})$ goes to
$+\infty$ when $\epsilon \to 0$
Moreover, $$\overline{H}_0=
\lim_{h\to 0}\lim_{\epsilon\to 0}\,  \frac{\lambda_{\epsilon,h} }{h}=0 $$
In this case
$$I(x,v)= L(x,v)+\nabla\phi_0(x)(v)-\overline{H}_{0}=  \frac{v^2}{2} ,$$
and the equation $I(x,v)=0$, means that, $v=0$. The Aubry set, as it
is known,  in this case is the set of elements of  the form $(x,0)$,
for any $x \in \mathbb{ T}^N.$

The Varadhan's Integral Lemma \cite{DZ} claims the following:
suppose $I(x,v) = I(v)$ is the deviation function for $\mu_{\epsilon
, h}$ as above, then, if $g: \mathbb{R}^N   \to \mathbb{R} $ is a
continuous function $g(v)$, then
$$ \lim_{\epsilon,h\rightarrow0}\epsilon \,  \ln
 \int \, e^{g(v) }   \mu_{\epsilon,h}(v)=
    \sup_v \{g(v) - I(v) \}\,=\,  \sup_v \{g(v) - \frac{v^2}{2} \}
$$

An interesting example is when $p$ is fixed and we consider $g(v) =
\langle p,v\rangle.$ In this case, $ \sup \{g(v) - \frac{v^2}{2} \}= p^2.$

\section{The discrete time Aubry-Mather problem}
\subsection{The uniqueness of the calibrated subactions }

In this section we will study some dynamical properties of the discrete
time Aubry-Mather problem
(see \cite{Gom}). These
will be used to obtain conditions for the uniqueness of  $\phi_h$ used
in the definition of $I_h$.

For a $h>0$ fixed, remember that
\[
\Hh_{h}=\min_{\mathcal M_h} \left\{ \int_{\tn\times\rn}
   L(x,v)d\mu(x,v)\right\}
\] where
$$\mathcal M_h=\left\{\mu\in\mathcal M\;:\;   \int_{\Tt^N\times \Rr^N}
\varphi(x+hv) -\varphi(x) d\mu=0, \forall \varphi\in C(\tn) \right\}.$$
A measure $\mu_h$ which attains such minimum is called a discrete Mather measure for $L$.
Note that  $\Hh_h$ (possibly up to a sign convention)   is the analog of Ma\~n\'e's critical value.

\bigskip

\begin{definition}A continuous function $u: \mathbb{T}^N \to \mathbb{R} $ is called

(a) a forward-subaction if
$$ u(x)\leq u(x+h\, v) + h\, L(x,v) -  h\, \Hh_h, \;\;\;\forall(x,v)\in\tn\times\rn,$$

(b) a backward-subaction if $$ u(x)\leq u(x-h\, v) + h\, L(x-hv,v) -  h\, \Hh_h, \;\;\;\forall(x,v)\in\tn\times\rn.$$
\end{definition}

\begin{definition}
A continuous function $u: \mathbb{T}^N \to \mathbb{R} $ is called a calibrated forward- subaction (calibrated subaction for short) if, for any $x$, we have
$$ u(x)= \inf_{v \in \mathbb{R}^N } \, \{ u(x+h\, v) + h\, L(x,v) -  h\, \Hh_h \}.
$$
\end{definition}

For each value $x$ this infimum is attained by some (can be more that one) $v(x)$.

\begin{definition}
A continuous function $u: \mathbb{T}^N \to \mathbb{R} $ is called a calibrated backward- subaction if, for any $x$ we have
$$ u(x)= \inf_{v \in \mathbb{R}^N } \, \{ u(x-h\, v) + h\, L(x-h v ,v) - h\, \Hh_h \}.$$
\end{definition}
By item (b) of theorem \ref{teo3}, any limit of a subsequence $ \lim_{\epsilon_{i} \to 0} \phi_{\epsilon_i,\, h} =  \phi_h$, is a calibrated subaction for $L$. In general it is not known if $\phi_h$ is unique (up to a constant).
We will establish bellow (Theorem \ref{unicidade}) a  condition for such uniqueness.
Similar properties are true for the backward problem, that is, if
$ \lim_{\epsilon_{i} \to 0} \bar \phi_{\epsilon_i,\, h} = \bar \phi_h$,
then $\bar \phi_h$ is a calibrated backward-subaction, etc...

\begin{proposition}\label{diff} Let $u$ be a calibrated subaction to the Lagrangian $L$.
If $u$ is differentiable at $x$ then
$$\nabla u(x)=hL_x(x,v(x))-L_v(x,v(x)). $$
\end{proposition}
This theorem can be shown using the same arguments  of the proof of  theorem 4.1 in \cite{Gom}.

\medskip

\textbf{Assumption:}  We shall suppose also that the Lagrangian is such that    $L_x$ has bounded  Lipschitz constant in $v$. Because in this case the equation  $p=hL_x(x,v(x))-L_v(x,v(x))$ has only one differentiable solution, when $h$ is small enough. Hence by the same arguments used in theorem 5.5 of \cite{Gom} we obtain that any minimizing measure  $\mu_h$ is supported in a graph.

\bigskip

The next definitions will be considered for a fixed   value of $h>0$, small enough, such that we have the graph property.

\begin{definition}
 Given $k$ and  $ x ,
z \in \rn  $, we will call a $k$-path
beginning in $x $ and ending at
$ z $ an ordered sequence of points
$$ (x_0,x_1,....,x_k)\in \rn\times ...\times \rn $$
satisfying $ x_0 = {x} $,  $x_k=z$.
\end{definition}
\hspace{.1in}

We will denote by
$ \mathcal P_k (x, z)=  \mathcal P_k^h (x, z)   $ the set of such $k$-paths.

\hspace{.1in}

\noindent For each $x_j$ we will associate a $v_j\in \mathbb {R}^N$, such that

$$   v_{ j} = \frac{x_{j + 1} -
x_{ j}}h, \mbox {\;\; for\;\; } 0\leq j< k    $$

\vspace{.4cm}


\vspace{.4cm}
\begin{definition}For a k-path fixed  $(x_0,...,x_k)$  we
define it action  by:
$$A_{L-\Hh_h}(x_0,...,x_k):=h\sum^{k-1}_{i=0}(L-\Hh_h)(x_i,v_i).$$
\end{definition}

\Rm Let $(x_0,...,x_k)\in  \mathcal P_k (x+s,  z) $  be a path, where $x,z\in \rn$ and  $s\in\zn$. As the Lagrangian is $\zn$-periodic we have that the path $(\tilde x_0,...,\tilde x_k)\in  \mathcal P_k (x,  z-s) $ given by $\tilde x_i=x_i-s$ is such that $A_{L-\Hh_h}(x_0,...,x_k)=A_{L-\Hh_h}(\tilde x_0,...,\tilde x_k).$

\begin{definition}
A point $x \in
\mathbb{T}^N $ is called non-wandering with respect to  $L $ if,
given  $ \epsilon > 0
$ there exist $k\geq 1$, $s_k\in\zn$ and a $k$-path
$ (x_0,...,x_k)\in  \mathcal P_k (
x+s_k,  x) $ such that
$$ \left |A_{L-\Hh_h}(x_0,...,x_k) \right | < \epsilon. $$
We will denote by $ \Omega_h(L) $ the set of non-wandering points with respect to $ L$.
\end{definition}

\Rm $ \Omega_h(L) $ is  a  closed set. Indeed, let  $x_k\in \Omega_h(L) $ be  such that $x_k\to x$. For each $x_k$ and $\epsilon=\frac{1}n$ there exists $j_n$, $s_{j_n}$ and $ (x_0,...,x_{j_n})\in \mathcal P_{j_n} (x_k+s_{j_n},  x_k)  $ such that $|A_{L-\Hh_h}(x_0,x_1,....,x_{j_n})|\leq \frac{1}n$. Hence the path $ (x+s_{j_n}, x_1,..., x_{j_n-1},x)$
has also small action, when $n\to\infty$ we get $x\in \Omega_h(L) $.
\bigskip

The proof for the results we describe bellow are similar to the ones   presented  in \cite{GL}
where the discrete time symbolic dynamics version of Aubry-Mather Theory is considered.

\begin {proposition}\label{projsupp} Let $\mu_h$ be a discrete-time Mather measure, then
$$\pi_1 (\supp(\mu_h))\subset\Omega(L)$$
\end {proposition}
 \begin {proof} By \cite{Gom} we know that $\mu_h$ is supported on a Lipschitz  graph, then we can define $\psi:\rn\to\rn$, such that $\psi(x)=x+hv(x)$, we define $\bar\psi:\tn\to\tn$ by $\bar\psi(x)=\psi(x)\mod\zn$.  We claim that $\mu\circ\pi_1^{-1}$ is $\bar\psi$-invariant.

  Indeed, as  $\mu_h$ is holonomic and by the definition of  $\psi$ we have that for all  $\phi:\tn\to \re$
  $$\int_{\tn}\phi\circ\bar \psi(x)d(\mu_h\circ\pi_1^{-1})=\int_{\tn\times\rn}\phi(x+hv(x)\mod \zn)d\mu_h=$$
 $$ =\int_{\tn\times\rn}\phi(x)d\mu_h=\int_{\tn}\phi(x)d(\mu_h\circ\pi_1^{-1})$$
Let $(x,v)\in \supp(\mu_h)$ and let B be an open ball centered at the point $x$,
 then $\mu\circ\pi_1^{-1}(B)>0$, hence there exists $x_0\in B$ such that $\bar\psi^j(x_0)$ returns infinitely many times to $B$, i.e., there exists $s_j\in\zn$ such that $\psi^j(x_0)-s_j\in B$.      Because $\phi_h$ is a calibrated subaction for L  we can write
 $\phi_h(x_0)-\phi_h(x_j)=h\sum_{i=0}^{j-1}(L-\Hh_h)(x_i,v_i)$, where $x_i:=\psi^i(x_0)$. Given $\delta>0$ and $x_j-s_j\in B$ we can construct the following path
$ (\tilde x_0,...,\tilde x_j)=(x,x_1,...x_{j-1},x+s_j)$
 such that
\[   A_{L-\Hh_h}(\tilde x_0,....,\tilde x_{j})    \leq\delta.\]
Indeed,
\begin{align*}
&A_{L-\Hh_h}(x_0,...,x_{j})=\phi_h(x_0)-\phi_h(x_j)
\\&\quad +h\left[ L(x,\frac{x_1-x}h)-L(x_0,v_0)+L(x_{j-1},\frac{x+s_j-x_{j-1}}h )-L(x_{j-1},v_{j-1})\right]\leq \delta,
\end{align*}
 if B is small enough. Hence $x\in \Omega_h(L).$
 \end{proof}

\begin{definition}
For a fixed value $h>0$,
define
$$ S_h^k(x, z) =\inf_{s\in\zn}
\inf_{ (x_0,...,x_k) \in  \mathcal P_k (x+s, z)  }
\; A_{L-\Hh_h}( x_0,...., x_{k}).$$
Let $S_h$ be the  Ma\~n\'e potential  the function
$ S_h :  \mathbb{T}^N \times \mathbb{T}^N  \to \mathbb R $ defined by
$$ S_h (x, z) = \inf_{k }  S^k_h(x, z). $$
The  Peierls barrier  $\bold{h}_h:\mathbb{T}^N \times \mathbb{T}^N  \to \mathbb R \cup \{ \pm \infty \} $
is the function defined by
$$ \bold{h}_h (x, z) = \liminf_{k\to\infty }  S^k_h( x, z). $$
\end{definition}
\vspace{.4cm}

Note that  $$\Omega_h(L)=\{x\in\tn\;:\;\bold{h}_h (x,  x)=S_h (x,  x)=0\}.$$

We point out here a main difference from the continuous time Mather problem where the  Ma\~n\'e
potential $S$ (defined in a similar way as for instance in \cite{Fa} or \cite{CI})  is zero for any pair $(x,x)$
where $x$ is in the configuration space. The point is that in the continuous
time  case we can consider trajectories with  time as small as  we want, whereas this is not possible in discrete time.

The functions  $S_h$ and $\bold{h}_h$ have the following properties:

\begin{itemize}
\item[(i)] $S_h(x,z)\leq S_h(x,y)+S_h(y,z)\;\;\forall x,y,z\in\tn$,

\item[(ii)] $\bold{h}_h (x,z)\leq \bold{h}_h (x,y)+\bold{h}_h (y,z)\;\;\forall x,y,z\in\tn.$

\item[(iii)] $\bold{h}_h (x,z)\leq S_h (x,y)+\bold{h}_h (y,z)\;\;\forall x,y,z\in\tn.$
\end{itemize}

\begin{proposition} Let us fix $z\in\tn$, the functions $S_h (\cdot,z)$ and $ \bold{h}_h (\cdot,z)$ are forward subactions.

\end{proposition}
\begin{proof} It follows by  (i) and (iii), respectively, and by the observation that
\[
S_h(x,y)\leq h\left[ L(x, \frac{y-x}{h})-\Hh_h\right].
\]
\end{proof}

In order to prove that $\bold{h}_h (\cdot,z)$ is a calibrated subaction, we need the following  lemma.
Also, note that if $z\in\Omega_h(L)$ then by (iii) we have that $\bold{h}_h (\cdot,z)$ is finite.

\begin{lemma} \label{minimal} Let  $(x_0,...,x_k)\in  \mathcal P_k (x+s, z) $ be  a path such that
$A_{L-\Hh_h}( x_0,...., x_{k})=S_h^k(x, z).   $
 Then there exists a constant $K$ such that  $|v_i|<K$ for all  $0\leq i<k$. Also, $K$ is independent of $x,z\in\tn$.
\end{lemma}
\begin{proof} Let $\displaystyle R=2\max_{x,y\in\tn}d(x,y)$, we define $\displaystyle A(R)=\max\{L(x,v)\; :\;|v|\leq R\}$. As $L$ is superlinear there exists $K$ such that if $|v|\geq K$ then $L(x,v)> A(R)$.

We will show the lemma by induction. First let us prove that $|v_0|<K$:  suppose by contradiction
 that $|x_1-x_0|>K$. We choose $s_0\in\zn$   such that $|x_0+s_0-x_1|<R$, then the  path $(\tilde x_0,...,\tilde x_k)=(x_0+s_0,x_1,...,x_k)$ is such that $A_{L-\Hh_h}(\tilde x_0,...,\tilde x_k)<A_{L-\Hh_h}( x_0,...., x_{k})=S_h^k(x, z)$, which is a contradiction. Suppose we have proved that $|v_i|<K$ for all $0\leq i<j$ and suppose by absurd that $|v_j|>K, $ we choose $s_{j-1}\in\zn$   such that $|x_{j-1}+s_{j-1}-x_j|<R$, then the path $(\tilde x_0,...,\tilde x_k)=(x_0+s_{j-1},...,x_{j-1}+s_{j-1},x_j,...,x_k)$ is such that $A_{L-\Hh_h}(\tilde x_0,...,\tilde x_k)<A_{L-\Hh_h}( x_0,...., x_{k})$, which is a contradiction, hence $|v_i|<K$, for all $0\leq i\leq j$.
 \end{proof}

\begin{proposition} \label{calibrada}For any   $z\in\Omega_h(L)$  the function  $u(\cdot) =  \bold{h}_h (\cdot,z)$ is a calibrated subaction.
\end{proposition}
\begin{proof}  For a point  $x\in\tn$, we want to find  $v\in\rn$ such that $$\bold{h}_h (x,z)-\bold{h}_h (x+hv,z)=hL(x,v)-h\Hh_h.$$
 By the definition of Peierls barrier there exist a sequence $j_n\to \infty$ and a sequence of paths
 $(x_0^n,..., x_{j_n}^n)\in\mathcal P_{j_n}(x+s_n,z)$, $s_n\in\zn$,  such that $$A_{L-\Hh_h}( x_0^n,...., x^n_{j_n})=S_h^{j_n}(x, z)\to\bold{h}_h(x,z).$$

As $|v_0^n|=|\frac{x+s_n-x_1^n}h|\leq K$, the sequence   $\{x_1^n-s_n\}$  has an accumulation point, say  $x_1$,  taking a subsequence if necessary, we can suppose that,   $\displaystyle x_1=\lim_{n}(x_1^{n}-s_n)$ and we define $\displaystyle v=\lim_{n}( \frac{x_1^n-x-s_n}h).$
Then, because $z\in \Omega_h(L)$,
$$\bold{h}_h (x_1,z)\leq S_h^{j_n-1}(x_1,z)\leq A_{L-\Hh_h}( x_1+s_n,x_2^n,...., x^n_{j_n}).$$
Hence
\begin{align*}
hL(x,v)-h\Hh_h+\bold{h}_h (x_1,z)&\leq\lim_{n\to\infty}\big[ hL(x+s_n,\frac{x_1^n-x-s_n}{h})-h\Hh_h\\
&\quad
+A_{L-\Hh_h}( x_1^n,...., x^n_{j_n}) \big]\\
& =\lim_{n\to\infty} A_{L-\Hh_h}( x_0^n,...., x^n_{j_n})     =\bold{h}_h(x,z).
\end{align*}
 Then
$$hL(x,v)-h\Hh_h\leq\bold{h}_h(x,z)-\bold{h}_h (x+hv,z).$$ As $\bold{h}_h(\cdot,z)$ is a subaction we have the other inequality. Hence $$hL(x,v)-h\Hh_h=\bold{h}_h(x,z)-\bold{h}_h (x+hv,z).$$\end{proof}

\Rm When $z\in\Omega_h(L)$ we have that $S_h(\cdot,z)=\bold{h}_h (\cdot,z). $

\begin{theorem}\label{teo12}
For a fixed value of $h$, if $u$ is a calibrated subaction, then for any $x$ we have
$$
u(x) = \inf_{p \in \Omega_h(L) }\{u(p) + S_h(x,p) \}.
$$
\end{theorem}
\begin{proof}
 By the definition of calibrated subaction  we have that
$$
u(x) \leq \inf_{p \in \Omega_h(L) }\{u(p) + S_h(x,p) \}.
$$
Let us now show
that $\displaystyle u(x) \geq \inf_{p \in \Omega_h(L) }\{u(p) + S_h(x,p) \}:$
Fix $x\in \mathbb{T}^N$, we will denote $x_0=x$. As  $u$ is a calibrated subaction there exists $v_0$ such that
\[
u(x_0)=u(x_0+h\, v_0) + h\, L(x_0,v_0) -  h\, \Hh_h.
\]
Let $x_1:=x_0+h\, v_0$, we can construct a sequence
$(x_0,x_1,...,x_j,...)$ such  that for each $j>0$, $x_{j+1}=x_j+h\, v_j$, and  $ u(x_j)=u(x_{j+1}) + h\, L(x_j,v_j) -  h\, \Hh_h $. We project this points  in the torus, i.e., we choose $s_j\in\zn$ such that  $\bar x_j=x_j+s_j\in\tn.$
		
Let $p\in \mathbb{T}^n$ be a limit point of the sequence $\{\bar x_j\}$,
we claim that  $p\in \Omega_h(L)$. Indeed, suppose $\bar x_{j_m}\to p$.  We
can construct, for $n>m$, the following path: $(\tilde
x_0,...,\tilde x_{j_n-j_m})=:(p-s_{j_m},x_{j_m+1},...,x_{j_n-1},p-s_{j_n})$, hence
$$ A_{L-\Hh_h}( \tilde x_0,....,\tilde  x_{j_n-j_m})=A_{L-\Hh_h}(  x_{j_m},...,  x_{j_n})+$$
$$ +h(L(p-s_{j_m},\frac{x_{j_m+1}-p+s_{j_m}}h)-
L(x_{j_m},v_{j_m}) +L(x_{j_n-1},\frac{p-s_{j_n}-x_{j_n-1}}h) -L(x_{j_n-1},v_{j_n-1})).
$$
  As $A_{L-\Hh_h}(  x_{j_m},...,  x_{j_n})=u(x_{j_m})-u(x_{j_n})$,
given $\epsilon>0$, if $m$ is large enough, then
$$|A_{L-\Hh_h}( \tilde x_0,....,\tilde  x_{j_n-j_m})|< \epsilon, $$
i.e., $p\in \Omega_h(L)$. For this $p$ let us show that $$S_h(x,p)\leq u(x)-u(p).$$

Indeed, we consider the following path:
$(\tilde x_0,...,\tilde x_{j_m})=(x_0,...,x_{j_m-1},p-s_{j_m})$, then
$$A_{L-\Hh_h}( \tilde x_0,....,\tilde  x_{j_m})-u(x)+ u(p)=$$
$$=u(p)-u(x_{j_m})-hL(x_{j_m-1},v_{j_m-1})+hL(x_{j_m-1},\frac{p-s_{j_m}-x_{j_m-1}}h). $$
Hence, given $k>0$ there exists $m_k\in\nat$ such that if $m>m_k$ then 
$$A_{L-\Hh_h}( \tilde x_0,....,\tilde  x_{j_m}) <u(x)- u(p)+1/k.$$
Finally, when $k\to \infty $ we obtain
 $$S_h(x,p)\leq u(x)-u(p),$$ and
$$
u(x) \geq \inf_{p \in \Omega_h(L) }\{u(p) + S_h(x,p) \}.
$$\end{proof}

\begin{proposition} $\mathcal O^{h}:=\{\psi\in C^{\infty}(\tn,\re)\;\;: \;\; \mathcal M_0^h(L+\psi)=\{\mu_h\} \mbox{ and }
\pi_1(\supp(\mu_h))=\Omega_h(L+\psi)\}$ is a generic set. Where $\mathcal
M_0^h$ denote the set of holonomic minimizing measures, i.e.,
probability
measures  in
$\Tt^n\times \Rr^n$
such that $\int L
d\mu_h =\Hh_h$ and $\int\varphi(x+hv)-\varphi(x)d\mu_h=0,
\,\,\,\forall \varphi\in C(\tn).$
\end{proposition}
\begin{proof}
The proof that $\mathcal O_{1}^{h}:=\{\psi\in C^{\infty}(\tn,\re)\;\;:\;\; \mathcal M_0^h(L+\psi)=\{\mu_h\}\} $
is generic is similar to the one
 in the continuous case, see \cite {CI}.

Let $\psi_0\in \mathcal O_{1}^{h}, $ and $\psi_1\in C^{\infty}(\tn,\re)$ such that $\psi_1\geq 0$ and $\{x: \psi_1(x)=0\}=\pi_1(\supp(\mu_h))$. Then $\pi_1(\supp(\mu_h))\subset \Omega_h(L+\psi_0+\psi_1)$.

Claim: If $x_0\notin\pi_1(\supp(\mu_h))$ then $x_0\notin\Omega_h(L+\psi_0+\psi_1)$. Indeed, $\psi_1(x_0)>0$, and
 $$\bold h_h^{(L+\psi_0+\psi_1)}(x_0,x_0)=\liminf _{k\to\infty}\left(\inf_{s\in\zn}\inf_{\mathcal P_k(x_0+s,x_0)}\sum _{i=0}^{k-1} (L+\psi_0+\psi_1-\Hh_h)(x_i,v_i)    \right)\geq$$
$$ \liminf _{k\to\infty}\left(\inf_{s\in\zn}\inf_{\mathcal P_k(x_0+s,x_0)}\sum _{i=0}^{k-1}(L+\psi_0-\Hh_h)(x_i,v_i)  +\psi_1(x_0)  \right)=\bold h_h^{(L+\psi_0)}(x_0,x_0)+\psi_1(x_0).$$Hence $\pi_1(\supp(\mu_h))=\Omega_h(L+\psi_0+\psi_1)$.
\end{proof}

\begin{proposition}\label{corresp}There exists a bijective correspondence between
the set of calibrated subactions and the set of functions
$f\in C(\Omega_h(L))$ satisfying $f( x)-f(z)\leq S_h(x,z)$,
 for all points $x,z$ in $\Omega(L)$.
\end{proposition}
The proof of this Proposition is similar to the proof of Theorem 13 in \cite{GL}.

\begin{proposition} \label{ergodica}Let   $\mu_h\circ \pi_1^{-1}$
be an  ergodic measure (with respect to the flow induced by $\bar \psi$),
and  $u,u'$ two calibrated subactions for $L$, then  $u-u'$ is constant in  $\pi_1(\supp(\mu_h))$.
\end{proposition}
\begin{proof}It was shown in \cite{Gom} that the points of the support of the measure  $\mu_h$ are the form  $(x,v)=(x_0+hv_0,v)$ with  $(x_0,v_0)$ in the support of $\mu_h$. Take  $x\in\pi_1\supp(\mu_h)$, then  $x=x_0+hv_0$, hence
$$u(x_0)-u(x_0+hv_0)=hL(x_0,v_0)-h\Hh_h=u'(x_0)-u'(x_0+hv_0)$$
Then  $u-u'=(u-u')\circ\bar \psi$
in  $\pi_1(\supp(\mu_h))$ and as  $\mu_h\circ\pi_1^{-1}$ is ergodic it follows that  $u-u'$ is constant in $\pi_1(\supp(\mu_h))$.\end{proof}

\begin{lemma} Suppose that  $L$ is generic and let $\mu_h$ be the unique minimizing measure, then the measure $\mu_h\circ \pi_1^{-1}$ is ergodic for the map  $\bar\psi$ (defined in the proof of the proposition   \ref{projsupp}).

\end{lemma}
\begin{proof} In proposition  \ref{projsupp} it was proved that  $\bar\psi$ is    $\mu_h\circ \pi_1^{-1}$-invariant. Let us show that
it is uniquely ergodic. Let  $\eta$ be a measure in the Borel sets of  $\Omega_h(L)=\pi_1(\supp(\mu_h))$,  invariant by  $\bar\psi:\Omega_h(L)\to\Omega_h(L)$. We define, for each  Borel set $B$ of  $\tn\times\rn$, $\mu(B)=\eta(\pi_1(B\cap\supp(\mu_h)))$, then  $\mu$ is a probability in $\tn\times\rn$, such that
\begin{itemize}
\item[(i)] $\supp(\mu)\subset\supp(\mu_h),$
\item[(ii)] $\mu\circ\pi_1^{-1}=\eta,$
\item[(iii)] $\mu\in\mathcal M_h$.
\end{itemize}
(i)
 and  (ii) are easily verified.   (iii): Let  $\varphi\in  C(\tn)$ be  a function, we have that

$$\int_{\tn\times\rn} \varphi(x+hv)d\mu(x,v)=\int_{\tn\times\rn} \varphi(x+hv(x)) d\mu(x,v) =\int_{\tn} \varphi\circ\bar \psi(x)d\eta(x)= $$
$$=\int_{\tn} \varphi(x)d\eta(x)=\int_{\tn\times\rn} \varphi(x)d\mu(x,v).$$

Let  $u$ be a calibrated subaction, by theorem 5.4 of  \cite{Gom} for each point  $(x,v)\in\supp(\mu_h)$ we have
$$ h L(x,v) = u(x)-u(x+hv)+h\Hh_h$$

By  (i) and  (iii) we have that
$$\int h L(x,v) d\mu(x,v)=\int (u(x)-u(x+hv)+h\Hh_h)\;d\mu(x,v)=h\Hh_h .$$
Hence  $\mu$ is a minimizing measure, but as we are supposing that minimizing measure is unique, we obtain  $\mu=\mu_h$. Therefore  $\eta=\mu_h\circ\pi_1^{-1}$, then $\mu_h$ is uniquely ergodic.\end{proof}

\begin{theorem}\label{unicidade}If  $L$ is generic in the Man\'e's sense, then the set of calibrated subactions is unitary  (up to constant).

\end{theorem}
\begin{proof} By hypothesis we have that  $\pi_1(\supp(\mu_h))=\Omega_h(L)$.

Let $f,f':\Omega_h(L)\to\re$ be continuous functions satisfying  $f( x)-f(\bar x)\leq S_h(x,\bar x)$, $f'( x)-f'(\bar x)\leq S_h(x,\bar x)$. As in proposition \ref{calibrada}, we construct two calibrated subactions $u_f, u_{f'}$  such that
$f-f'=u_f-u_{f'}$ in  $\Omega_h(L)$, now by the proposition  \ref{ergodica} $u_f-u_{f'}$ is constant in
$\pi_1(\supp(\mu_h))=\Omega_h(L)$. Hence the set  $\{f\in C(\Omega_h(L)\;:\;f( x)-f(\bar x)\leq S_h(x,\bar x) \}$ is unitary, by proposition  \ref{corresp} we conclude that the set of calibrated subactions is unitary.\end{proof}


\Rm Note that the definition of the Lagrangian be generic depends on the property  $P$ we consider. We fix a sequence $h_n\to 0$, for each $h_n$ we consider the  property  $P_n$ given by: $\mathcal M_0^{h_n}(L+\psi)=\{\mu_{h_n}\} \mbox{ and }
\pi_1(\supp(\mu_{h_n}))=\Omega_{h_n}(L+\psi)$ .  Then, for each $n$, there exists a generic set  $\mathcal O^n\in C^{\infty}(\tn,\re)$ where  $P_n$ is verified.

We define
$$ \mathcal O=\bigcap_{n\geq 0}  \mathcal O^n$$
Hence, if $\psi\in \mathcal O$ then  $L+\psi$ has the property    $P_n$ for each  $n$.

\begin{corollary}Suppose that the   Lagrangian $L$  satisfy the hypothesis (1) to (3),
and is generic in the sense of the previous remark. Let $u_h(\cdot)=\bold{h}_h(\cdot, x)$, where $x\in \Omega_{h}(L)$, define $\tilde u_h=u_h -u_h(0)$. Then $\tilde u_h$ converges to the unique viscosity solution $\phi_0$ of the H-J equation, which can be show to be $h(\cdot,z)$, where $z\in \mathcal A$, and $h$ is the Peierls barrier.
\end{corollary}
\begin{proof}
The hypothesis (3) implies that $u_h$ is semiconcave (uniformly in $h$) and hence locally Lipschitz. Thus, by periodicity $\tilde u_h$ is an uniformly bounded and equicontinuous
family.
It follows by theorem \ref{unicidade}, proposition \ref{calibrada} and item (b) of theorem \ref{teo5}.
\end{proof}
Here we finish the part strictly necessary for the results
required by the first part of the paper.
\subsection{Existence of a separating subaction}
In this last part we are interested in showing a discrete  analog
of the \cite{FS}, that is the existence of a separating subaction,
as in \cite{GLT}.  We add  Theorem (\ref{FS}) in
order to have a more complete understanding of the Discrete Time
Aubry-Mather Problem.

 For this goal we need to consider the
Hamiltonian defined in the following way.

\begin{definition} Let $L(x,v):\tn\times\rn\to \re$ be the Lagrangian, we define $$\tilde{H}(p,x)=\max_v \{p\cdot v - L(x,v)\}.$$
\end{definition}

The equation
$$\max_v \left\{\frac{u(x+hv)-u(x)}{h}-L(x,v)\right\}\leq -\Hh_h$$
can be seen as a discrete  analogous  of the  Hamilton-Jacobi equation
$$\max_v \{\nabla u(x)\cdot v - L(x,v)\}=\tilde{H}(\nabla u(x),x)\leq -\Hh_0.$$

\begin{definition}For a fixed value $h>0$, a continuous function
  $u:\tn\to\re$ is called a subaction if for all  $x\in\tn$ we have
$$\max_v \left\{{u(x+hv)-u(x)}-hL(x,v)\right\}\leq -h\Hh_h.$$
\end{definition}

\begin{definition}
We say that a  subaction $u$  is separating if
$$\max_v \left\{u(x+hv)-u(x)-hL(x,v)\right\}= -h\Hh_h \iff x\in\Omega_h(L).$$
\end{definition}

Our main result of this last part is the following:
\begin{theorem}\label{FS} There exists a separating subaction.
 \end{theorem}
Before proceeding with the proof, we need some preliminary results.
\vspace{0.3cm}
\begin{lemma}\label{igualdade} For any subaction $u$ and all
$x\in\Omega_h(L)$ we have
$$\max_v \left\{u(x+hv)-u(x)-hL(x,v)\right\}= -h\Hh_h.$$
\end{lemma}
We will postpone the proof of the Lemma.

From now on we will suppose  $h=1$, and $\Hh:=\Hh_1$ (here we don't need the graph property).

Note that the definition of subaction
$$\max_v \left\{u(x+v)-u(x)-L(x,v)\right\}\leq -\Hh$$
is equivalent to
\begin{equation}\label{subacao}
u(x_k)-u(x_0)\leq A_{L-\Hh}(  x_0,....,  x_{k})
\quad \text{for any path}\  (x_0,...,x_k).
\end{equation}
 By this  characterization of the subactions, it is easy to see that
 $\bold{h}_x=\bold{h} (x,  \cdot)$ and $S_x=S (x,  \cdot)$ are subactions.

\begin{proposition} \label{peierl} If $x\in\Omega(L)$ there exists a sequence
  $(x_0,x_1,....,x_k,...)$ such that  $x_0=x$ and for all $k$
$$\bold{h}(x_k,x_0)\leq -A_{L-\Hh}(  x_0,....,  x_{k}).$$
\end{proposition}
\begin{proof}
Since  $x\in\Omega(L)$ there exists a sequence of
minimal paths  $\{(x^n_0,...,x^n_{j_n})\}_{n\in\Nn}$
such that $x^n_0=x$, $x^n_{j_n}=x+s_{j_n}$ and $j_n\to\infty$ satisfying
\begin{equation}\label{somatend0}
A_{L-\Hh}(  x_0^n,....,  x_{j_n}^n) \to 0.
\end{equation}
As  $|v^n_j|<K$
there exists a sequence
  $(x_0,...,x_k,...)$ which  is the limit of the paths above,
the convergence being uniform in each compact part.

Fixed $k\in\Nn$, for $j_n>k$ we have that
$$S^{j_n-k}(x_k,x_0)\leq L(x_k,x^n_{k+1}-x_k)-\Hh+A_{L-\Hh}(  x_{k+1}^n,....,  x_{j_n}^n) ,$$
and so
$$ S^{j_n-k}(x_k,x_0)-A_{L-\Hh}(  x_0^n,....,  x_{j_n}^n)
\leq  L(x_k,x^n_{k+1}-x_k)-\Hh-A_{L-\Hh}(  x_0^n,....,  x_{k+1}^n) .$$
Hence taking the $\displaystyle\lim_{n\to \infty}$ and using (\ref{somatend0}) we obtain

$$\bold{h}(x_k,x_0)\leq -A_{L-\Hh}(  x_0,....,  x_{k}).$$
\end{proof}

 \begin{proof}{(\textit{of Lemma \ref{igualdade}})}
It follows from (\ref{subacao}) that if $u$ is a subaction, then
$u(\bar y)-u( y)\leq\bold{h}( y,\bar y)$.

Let $x\in\Omega(L)$ and $(x_0,...,x_k,...)$ be the sequence given by
 proposition (\ref{peierl}).
If $u$ is a subaction, by Proposition (\ref{peierl}) we have
$$u(x_0)-u(x_k)\leq \bold{h}(x_k,x_0)\leq -A_{L-\Hh}(  x_0,....,  x_{k}).$$
The other inequality follows from (\ref{subacao}), hence
$$u(x_k)-u(x_0)=  A_{L-\Hh}(  x_0,....,  x_{k}),$$
in particular, for  $k=1$, this implies

$$\max_v[u(x+v)-u(x)-L(x,v)]= -\Hh.$$
\end{proof}

\begin{lemma}\label{sx} The function  $S_x(\cdot)=S(x,\cdot)$ is uniformly  Lipschitz in $x$.

\end{lemma}

\begin{proof} We fix $x\in\tn, \epsilon>0$ and  $y,z\in\tn$. By the definition of  $S$ there exists a path  $(x_0,...,x_k)\in\mathcal P_k(x+s,y), s\in \zn$ such that
$$|A_{L-\Hh}(  x_0,....,  x_{k}) |\leq S(x,y)+\epsilon,$$  we can construct the following path  $(\tilde x_0,...,\tilde x_k)=(x_0,...,x_{k-1},z)\in\mathcal P_k(x+s,z)$, the action of such path is given by  $$A_{L-\Hh}( \tilde x_0,....,  \tilde x_{k})=A_{L-\Hh}(  x_0,....,  x_{k})+L(x_{k-1},z-x_{k-1})-L(x_{k-1},y-x_{k-1}),$$
Note that $|y-x_{k-1}|\leq K$ and as $y,z\in\tn$, for any  $\theta\in(0,1)$, we have that  $|z-x_{k-1}+\theta(y-z)|<K_1$, for any  $x_{k-1}$, hence   $$|L(x_{k-1},z-x_{k-1})-L(x_{k-1},y-x_{k-1})|\leq\max_{(x,v)\in\tn\times K_1} |L_v(x,v)|\;|z-y|=C|z-y|$$
Then for all  $\epsilon>0$ we have that
$$S(x,z)\leq A_{L-\Hh}( \tilde x_0,....,  \tilde x_{k})\leq S(x,y)+\epsilon + C|z-y| $$ Which implies  $S(x,z)-S(x,y)\leq  C|z-y|$. Changing the roles of  $y$ and $z$ we get  $S(x,y)-S(x,z)\leq  C|z-y|$.

Therefore  $|S_x(y)-S_x(z)|\leq  C|z-y|$, note that the Lipschitz constant is independent of   $x$.\end{proof}

\begin{proof} (\textit{of theorem \ref{FS}}) Remember that the function  $S_x(\cdot)=S(x,\cdot)$ is a subaction.

By the definition of  $S$ we have that  $$S(x,x+v)\leq L(x,v)-\Hh \;\;\;\;\forall\;v$$

Fix $x\notin \Omega(L)$, then  $S(x,x)>0$. Hence
$$S_x(x+v)-S_x(x)<L(x,v)-\Hh   \;\;\;\;\forall \;v$$

As  $\Omega(L)$ is closed, for each $x\notin \Omega(L)$  we can
find a neighborhood   $V_x$ of $x$
such that for all $y\in V_x$

$$S_x(y+v)-S_x(y)<L(y,v)-\Hh  \;\;\;\;\forall \;v$$

We can extract from this family of neighborhoods  $\{V_x\}_{x\notin\Omega(L)}$, a countable subcover $\{V_{x_j}\}_{j=1}^{\infty}$. And we define
$$\tilde S_{x_j}(z)=S_{x_j}(z)-S_{x_j}(0),$$ as $S_{x_j}$ is uniformly  Lipschitz we obtain that  $|\tilde S_{x_j}(z)|\leq C |z|$, hence the series given by

$$u(z)=\sum_{j=1}^{\infty}\frac{\tilde S_{x_j}(z)}{2^j}$$ is well defined and uniformly convergent, as  $\tn$ is a compact set.
Finally we show that
  $u$ is a subaction:  $$u(x+v)-u(x)=\sum_{j=1}^{\infty}\frac{\tilde S_{x_j}(x+v)-\tilde S_{x_j}(x)}{2^j}=\sum_{j=1}^{\infty}\frac{ S_{x_j}(x+v)- S_{x_j}(x)}{2^j}\leq $$
  $$\leq\sum_{j=1}^{\infty} \frac{L(x,v)-\Hh}{2^j}=L(x,v)-\Hh.$$

  Hence by the theorem  (\ref{igualdade})

$$\max_v \left\{u(x+v)-u(x)-L(x,v)\right\}= -\Hh\;\; \mbox{\;\; if \;\;}\;\; x\in\Omega(L)$$
 and for  $x\notin \Omega(L)$, there exists  $k\geq 1$ such that $x\in V_{x_k}$ , hence
 $$S_{x_k}(x+v)-S_{x_k}(x)<L(x,v) -\Hh \;\;\;\;\;\forall\; v.$$

Therefore $$u(x+v)-u(x)=\left(\sum_{j\neq k}\frac{S_{x_j}(x+v)-S_{x_j}(x)}{2^j}+\frac{S_{x_k}(x+v)-S_{x_k}(x)}{2^k}\right)$$
$$<\left(\sum_{j\neq k}\frac{(L(x,v) -\Hh)}{2^j}+\frac{L(x,v) -\Hh}{2^k}\right)=L(x,v) -\Hh \;\;\;\forall \;v$$ i.e.,

$$\max_v \left\{u(x+v)-u(x)-L(x,v)\right\}< -\Hh \mbox{\;\;\;if\;\;\;} x\notin \Omega(L).$$\end{proof}

The present work is part of the PhD thesis of the last author in
"Programa de P\'os-Gradua\c c\~ao em Matem\'atica" - UFRGS (Brasil).


\begin{thebibliography}{99}

   \vspace{0.6cm}





 \bibitem [A1] {A1} N. Anantharaman. {\it On the zero-temperature or vanishing viscosity
limit for certain Markov processes arising from Lagrangian dynamics}. J. Eur. Math. Soc.,
 6 (2): 207-276, 2004.

\bibitem [A2] {A2} N. Anantharaman. {\it Counting geodesics which are optimal in homology}. Erg. Theo. and Dyn. Syst.,
   23 (2): 353-388, 2003.



\bibitem [AIPS] {AIPS}  N. Anantharaman, R. Iturriaga, P.  Padilla, H. Sánchez-Morgado,
{\it Physical solutions of the Hamilton-Jacobi equation.}
Disc. Contin. Dyn. Syst. Ser. B 5 (3) 513-528,  2005.




 \bibitem [BLT] {BLT} A. Baraviera, A. O. Lopes and Ph. Thieullen.
{\it A large Deviation Principle for equilibrium states of Holder potentials: the zero temperature case}. Stoch. and  Dyn.
  (6): 77-96, 2006.



 \bibitem [CS] {CS} P. Cannarsa and C. Sinestrari. {\it
     Semiconcave functions, Hamilton-Jacobi equations, and optimal
     control}. Progress in Nonlinear Differential Equations and their
   Applications, 58. Birkhäuser Boston Inc., Boston, MA, 2004.





 \bibitem [DZ] {DZ} A. Dembo and O. Zeitouni. {\it  Large Deviations Techniques and Applications}. Springer Verlag, 1998.

\bibitem [CDG] {CDG} F.Camilli, I.C. Dolcetta and D.A. Gomes {\it Error estimates for the approximation of the effective Hamiltonian.} preprint - to appear.

 \bibitem [CI] {CI} G Contreras and R. Iturriaga.  {\it Global Minimizers of
 Autonomous    Lagrangians}. AMS  2004. To appear.

\bibitem [CP] {CP} G. Contreras and G. Paternain. {\it Connecting orbits between static classes for generic Lagrangian systems}. Topology 41, pp 645-666, 2002.

 \bibitem [Fa] {Fa} A. Fathi.  {\it Weak KAM Theorem and
     Lagrangian Dynamics}. Cambridge University Press 2004. To appear.


 \bibitem [FS] {FS} A. Fathi and A. Siconolfi. {\it Existence of   $\, C^1$ critical
 subsolutions of the Hamilton-Jacobi equations}. Inv. Math. 155, pp 363-388, 2004.


 \bibitem [Gom] {Gom} D. A. Gomes. {\it Viscosity Solution methods
     and discrete Aubry-Mather problem}. Discrete Contin. Dyn. Syst.,
   13(1): 103-116, 2005.


 \bibitem [GL] {GL} E. Garibaldi and A. O. Lopes. {\it On Aubry-Mather theory for symbolic Dynamics}. Erg Theo and Dyn Systems, Vol 28 , Issue 3, 791-815, 2008.



 \bibitem [GLT] {GLT} E.Garibaldi, A. O. Lopes and P. Thieullen. {\it On separating sub-actions}. Preprint 2006. To appear.

 \bibitem [GL1] {GL1} D. A. Gomes and A. O. Lopes. {\it Exponential Decay of correlation for the Stochastic Process
 associated to the Entropy Penalized method}. Preprint 2007. To appear in S\~ao Paulo Journal of Mathematical Sciences.


 \bibitem [GV] {GV} D. A. Gomes and E. Valdinoci. {\it Entropy
     Penalization Methods for Hamilton-Jacobi Equations}. Adv. Math. 215, No. 1, 94-152, 2007.



\bibitem[Man] {Man}
R. Ma\~n\'e. {\it  Generic properties and problems of minimizing
measures of Lagrangian systems}. {Nonlinearity}, N. 9, 273-310 1996.


\bibitem[Mat] {Mat} J. Mather. {\it Action minimizing invariant measures for positive definite Lagrangian systems}. Math. Z., N 2, 169-207, 1991.

\end{thebibliography}
\end{document}